\documentclass[reqno,final,11pt]{amsart}
\usepackage{times}
\usepackage[dvips]{graphicx}
\usepackage{color}
\usepackage{amssymb,amsthm,amsmath,color,fullpage,url}
\usepackage{enumerate}
\usepackage{todonotes}
\numberwithin{equation}{section}

\newtheorem{theorem}{Theorem}[section]

\newtheorem{lemma}[theorem]{Lemma}

\newcommand{\oR}{\mathbb{R}}

\newcommand{\sgn}{\operatorname{sgn}}

\newcommand{\mat}[4]{[#1, #2; #3, #4]}

\DeclareMathOperator{\link}{link}
\DeclareMathOperator{\kr}{cr}

\DeclareMathOperator{\sign}{sign}

\DeclareMathOperator{\Conf}{Conf}

\begin{document}
\title{On $2$-cycles of graphs}

\author{Hein van der Holst}
  \address{Department of Mathematics and Statistics, Georgia State University, Atlanta, Georgia , USA}
\author{Serguei Norine}
 \address{Department of Mathematics and Statistics,  McGill University, Montreal, QC, H3A 2K6, Canada}
\author{Robin Thomas}
 \address{School of Mathematics, Georgia Institute of Technology, Atlanta, Georgia 30332-0160, USA}

\begin{abstract}
Let $G=(V,E)$ be a finite undirected graph. Orient the edges of $G$ in an arbitrary way. A $2$-cycle on $G$ is a function $d : E^2\to \mathbb{Z}$ such for each edge $e$, $d(e, \cdot)$ and $d(\cdot, e)$ are circulations on $G$, and $d(e, f) = 0$ whenever $e$ and $f$ have a common vertex. We show that each $2$-cycle is a sum of three special types of $2$-cycles: cycle-pair $2$-cycles, Kuratowski $2$-cycles, and quad $2$-cycles. In the case that the graph is Kuratowski connected, we show that each $2$-cycle is a sum of cycle-pair $2$-cycles and at most one Kuratowski $2$-cycle. Furthermore,  if the graph is Kuratowski connected, we characterize when every Kuratowski $2$-cycle is a sum of cycle-pair $2$-cycles. 
A $2$-cycle $d$ on $G$ is skew-symmetric if $d(e,f) = -d(f,e)$ for all edges $e,f\in E$.
We show that each $2$-cycle is a sum of two special types of skew-symmetric $2$-cycles: skew-symmetric cycle-pair $2$-cycles and skew-symmetric quad $2$-cycles.  In the case that the graph is Kuratowski connected, we show that each skew-symmetric $2$-cycle is a sum of skew-symmetric cycle-pair $2$-cycles. Similar results like this had previously been obtained by one of the authors for symmetric $2$-cycles. Symmetric $2$-cycles are $2$-cycles $d$ such that $d(e,f)=d(f,e)$ for all edges $e,f\in E$.
\end{abstract}

\maketitle


\section{Introduction}\label{sec:intof}

The \emph{configuration space}, $\Conf(X, k)$, of a topological space $X$ is the space
\begin{equation*}
\Conf(X, k) = \{(x_1,\ldots,x_k)\in X\times \cdots \times X : x_i\not=x_j\text{ for }i\not=j\}.
\end{equation*}
The configuration space $\Conf(X, k)$ was introduced by Fadell and Neuwirth \cite{Fadell_Neuwirth_1962}, and describes all collision-free positions of $k$ point objects in $X$, and it can  be used to describe possible collision-free trajectories of $k$ point objects in $X$. The study of the topology of $\Conf(X, k)$ is therefore of interest.
The configuration space $\Conf(X, k)$ has been studied by many authors. See Fadell and Neuwirth \cite{Fadell_Neuwirth_1962}, Arnol'd \cite{MR242196}, Vassiliev \cite{MR1168473}, and Totaro \cite{MR1404924}.

In the case that $G$ is a graph, viewed as a topological space, $\Conf(G, 2)$ describes all collision-free positions of two points in the graph $G$. As stated above, $\Conf(G,2)$ can be used to describe possible collision-free trajectories of $2$ point objects in $G$. In this paper, we study the second homology group $H_2(\Conf(G, 2))$. Results on $\Conf(G, 2)$ have been obtained by Patty \cite{MR126274, MR143211}, Copeland and Patty \cite{CopPat70}, Ghrist \cite{MR1873106}, Ghrist and Koditschek \cite{MR1882808}, Farber \cite{MR2455573}, and 
Barnett and Farber \cite{barnett2009}. Copeland and Patty~\cite{CopPat70} obtained upper and lower bounds on the Betti numbers of $\Conf(G, 2)$. 

In this paper, we study the group of $2$-cycles on $G$. This group is isomorphic to the homology group $H_2(\Conf(G, 2))$. We show that each $2$-cycle is a sum of three special types of $2$-cycles: cycle-pair $2$-cycles, Kuratowski $2$-cycles, and quad $2$-cycles. In the case that the graph is Kuratowski connected, we show that each $2$-cycle is a sum of cycle-pair $2$-cycles and at most one Kuratowski $2$-cycle. Furthermore,  if $G$ is Kuratowski connected, we characterize when every Kuratowski $2$-cycle is a sum of cycle-pair $2$-cycles. We show similar results for skew-symmetric $2$-cycles. Similar results had previously been obtained for symmetric $2$-cycles by Van der Holst \cite{Holst2003b}.

Let $G=(V,E)$ be a finite undirected graph.  Orient the edges of $G$ in an arbitrary way.
For each vertex $v$ and each oriented edge $e$ of $G$, define 
\begin{equation*}
[v,e] = \begin{cases}
+1 & \text{if $e$ is oriented towards $v$,}\\
-1 & \text{if $e$ is oriented from $v$,}\\
0 & \text{if $e$ is not incident with $v$.}
\end{cases}
\end{equation*}
A \emph{circulation} on $G$ is a function $f : E\to \mathbb{Z}$ such that 
\begin{equation*}
\sum_{e\in \delta(v)}[v,e]f(e) = 0
\end{equation*} 
for each vertex $v$ of $G$. If $C$ is an oriented cycle of $G$, then $C$ defines a circulation $\chi_{C} : E\to \mathbb{Z}$ by 
\begin{equation*}
\chi_{C}(e) = \begin{cases}
 1 & \text{if $e$ is traversed in forward direction by $C$},\\
 -1& \text{if $e$ is traversed in backward direction by $C$, and}\\ 
 0 & \text{otherwise;}
 \end{cases}
 \end{equation*} 
 we call $\chi_{C}$ a circulation on $C$.

A \emph{$2$-cycle} on $G=(V,E)$ is a function $d: E^2\to \mathbb{Z}$ such that 
\begin{enumerate}[(i)]
\item $d(e, f) = 0$ whenever $e$ and $f$ share a vertex, and 
\item for each $e\in E$, $d(e,\cdot)$ and $d(\cdot, e)$ are circulations on $G$.
\end{enumerate}
We denote the group of all $2$-cycles on $G$ by $L(G)$.

Let $H=(W, F)$ be a subgraph of $G=(V,E)$. If $d$ is a $2$-cycle on $H$, define the function $d' : E^2\to \mathbb{Z}$ by $d'(e,f) = d(e,f)$ for all $e,f\in F$ and $d'(e,f) = 0$ if $\{e,f\}\not\subseteq F$. Then $d'$ is a $2$-cycle on $G$. We can, therefore, consider any $2$-cycle on a subgraph $H$ of a graph $G$ as a $2$-cycle on $G$.

An example of a $2$-cycle on $G=(V,E)$ is the following. For oriented cycles $C$ and $D$ of $G$, let $d_{C,D} : E^2\to \mathbb{Z}$ be defined by 
\begin{equation*}
d_{C,D}(e, f) = \chi_C(e)\chi_D(f)
\end{equation*} 
for all $e,f\in E$. If $C$ and $D$ are disjoint, we call such a $2$-cycle a \emph{cycle-pair $2$-cycle}. Observe that if $C$ and $D$ are not disjoint, then $d_{C, D}$ is not a $2$-cycle as the condition that $d(e, f) = 0$ whenever $e$ and $f$ share a vertex is not satisfied. (In this paper, when we write disjoint, we mean vertex-disjoint.) We denote the subgroup of $L(G)$ generated by the set of all cycle-pair $2$-cycles by $B(G)$. 

A \emph{Kuratowski subgraph} of $G$ is a subgraph homeomorphic to $K_5$ or $K_{3,3}$. Let $V(K_5)=\{v_1,v_2,v_3,v_4,v_5\}$. Define 
\begin{equation*}
K(v_iv_j, v_k v_l):=\sgn(ijklm)
\end{equation*} 
for all permutations $(i,j,k,l,m)$ of $\{1,2,3,4,5\}$. It is easy to verify that $K$ is a $2$-cycle on $K_5$. In fact, every $2$-cycle on $K_5$ is an integral multiple of $K$.
Similarly, for $K_{3,3}$, let $V(K_{3,3})=\{a_1,a_2,a_3,b_1,b_2,b_3\}$ with the bipartition $(\{a_1,a_2,a_3\},\{b_1,b_2,b_3\}$. Define 
\begin{equation*}
K(a_{i_1}b_{j_1}, a_{i_2}a_{j_2}):=\sgn(i_1i_2i_3)\sgn(j_1j_2j_3)
\end{equation*} 
for all permutation $(i_2,i_2,i_3),(j_1,j_2,j_3)$ of  $\{1,2,3\}$. Then $K$ is a $2$-cycle on $K_{3,3}$. The $2$-cycles described in this paragraph are called \emph{elementary Kuratowski $2$-cycles}.

Let $G'=(V',E')$ be a graph obtained from a graph $G=(V,E)$ by replacing an edge $e \in E$ by a path consisting of edges $e_1,e_2,\ldots,e_k$. Orient the edges $e_1,e_2,\ldots,e_k$ in the same direction as $e$. Given a function $K : E^2\to \mathbb{Z}$, we define the function $K' : E'^2\to \mathbb{Z}$ by setting $K'(f_1, f_2)=K(f_1, f_2)$ if $f_1,f_2 \not \in \{e_1,e_2,\ldots, e_k\}$, $K'(e_i, f)=K(e, f)$, $K'(f, e_i)=K(f, e)$ and $K(e_i, e_j)=K(e, e)$ for $i,j \in \{1,2,\ldots,k\}$. We say that $K'$ is a \emph{subdivision} of $K$. 
A \emph{Kuratowski $2$-cycle} is a subdivision of an elementary Kuratowski $2$-cycle. Note that by Kuratowski's theorem a graph is non-planar if and only if there exists a Kuratowski $2$-cycle on $G$.

It was a folklore conjecture that every $2$-cycle on $G$ is a sum of an element in $B(G)$ and Kuratowski $2$-cycles on $G$. This was disproved by Barnett \cite{Barnett2010a}: there exists a $2$-cycle on $K_{3,4}$ which is not a sum of an element in $B(G)$ and Kuratowski $2$-cycles. However, we will see that the conjecture holds in case the graph is sufficiently connected. Another case where the conjecture holds is for symmetric $2$-cycles. 

If $d$ is a $2$-cycle on a graph $G=(V,E)$, we denote by $T(d)$ the $2$-cycle on $G$ defined by $T(d)(e,f) = d(f,e)$ for all edges $e,f\in E$.

A $2$-cycle $d$ on a graph $G=(V, E)$ is \emph{symmetric} if $d(e, f) = d(f, e)$ for all edges $e, f\in E.,$ that is, if $T(d) = d.$ We denote the group of all symmetric $2$-cycles on $G$ by $L^{\text{sym}}(G)$. For disjoint oriented cycles $C$ and $D$ of $G$, we call $d_{C,D} + d_{D,C}$ a \emph{symmetric cycle-pair $2$-cycle}. We denote the subgroup of $L^{\text{sym}}(G)$ generated by the set of all symmetric cycle-pair $2$-cycles by $B^{\text{sym}}(G).$ Observe that Kuratowski $2$-cycles are symmetric.
In \cite{Holst2003b}, van der Holst showed that every symmetric $2$-cycle on $G$ is a sum of an element in $B^{\text{sym}}(G)$ and Kuratowski $2$-cycles on $G$. 

A $2$-cycle $d$ on a graph $G=(V, E)$ is \emph{skew-symmetric} if $d(e, f) = -d(f, e)$ for all edges $e, f\in E,$ that is, if $T(d)=-d.$ We denote the group of all skew-symmetric $2$-cycles on $G$ by $L^{\text{skew}}(G)$. For disjoint oriented cycles $C$ and $D$ of $G$, we call $d_{C,D} - d_{D,C}$ a \emph{skew-symmetric cycle-pair $2$-cycle}. We denote the subgroup of $L^{\text{skew}}(G)$ generated by the set of all skew-symmetric cycle-pair $2$-cycles by $B^{\text{skew}}(G)$.

As cited above, Barnett proved that on $K_{3,4}$, there exists a $2$-cycle that is not a sum of cycle-pair $2$-cycles and Kuratowski $2$-cycles. The $2$-cycle introduced by Barnett is a special case of quad $2$-cycles; in Section~\ref{sec:quad}, we will give the definition. Our main results are: 

\begin{theorem}\label{thm:decomp}
Let $G$ be a graph. Then $L(G)$ is spanned by $B(G)$, the Kuratowski $2$-cycles, and the quad $2$-cycles of $G$.
\end{theorem}  

If $q$ is a quad $2$-cycle, then $q-T(q)$ is called a \emph{skew-symmetric quad} $2$-cycle.

\begin{theorem}
Let $G$ be a graph. Then $L^{\text{skew}}(G)$ is spanned by $B^{\text{skew}}(G)$ and the skew-symmetric quad $2$-cycles.
\end{theorem}

The similar theorem for symmetric matrices was proved in \cite{Holst2003b}.
\begin{theorem}
Let $G$ be a graph. Then $L^{\text{sym}}(G)$ is spanned by $B^{\text{sym}}(G)$ and the Kuratowski $2$-cycles.
\end{theorem} 

Let $G=(V,E)$ be a graph, let $H$ be a subgraph of $G$, and let $S$ be a subset of $V$ or a subgraph of $G$. By $N_H(S)$ we denote the set of vertices $u$ in $H-S$ that are adjacent to a vertex in $S$. If $H := G-S$, we write $N(S)$ for $N_H(S)$. For example, if $e=uv$ is an edge of $G$, and $H=G-\{u,v\}$, then $N(\{u,v\})$ is the set of vertices in $H$ that are adjacent to $u$ or $v$.

Let $G=(V, E)$ be a graph and let $F_1, F_2\subseteq E$. If $d : E^2\to \mathbb{Z}$, we denote by $d_{\restriction F_1\times F_2}$ the restriction of $d$ to $F_1\times F_2$.

The outline of the proof of Theorem~\ref{thm:decomp} is as follows. Let $d$ be a $2$-cycle on a graph $G=(V,E)$. In Section~\ref{sec:conn}, we show that we may assume that $G$ is $3$-connected. By a theorem of Thomassen (see \cite{Diestel}) there exists an edge $e=uv$ such that $G/e$ is $3$-connected unless $G$ has four vertices, in which case the theorem is clear. We will show that we can subtract cycle-pair $2$-cycles, quad $2$-cycles, and Kuratowski $2$-cycles from $d$, resulting in a $2$-cycle $d'$ so that $d'(f, g) = 0$ and $d'(g, f) = 0$ for every $f\in \delta(u)$ and every $g\in \delta(v)$. To show this, we show that we can subtract cycle-pair $2$-cycles and a quad $2$-cycle from $d$ yielding a $2$-cycle $d'$ so that $d'(f, g) = 0$ and $d'(g, f) = 0$ for every $f\in \delta(u)$ and every $g\in \delta(v)$, unless the vertices in $N(\{u,v\})$ are on a cycle $C$. We then show that if there are disjoint paths in $H:=G-\{u,v\}$ with ends $r_1,r_2\in N(\{u,v\})$ and $s_1,s_2\in N(\{u,v\})$, respectively, such that $r_1,s_1,r_2,s_2$ occur in $C$ in order, then  we can subtract cycle-pair $2$-cycles and a quad $2$-cycle from $d$ yielding a $2$-cycle $d'$ so that $d'(f, g) = 0$ and $d'(g, f) = 0$ for every $f\in \delta(u)$ and every $g\in \delta(v)$. We may therefore assume that there are no paths in $H$ with ends $r_1,r_2\in N(\{u,v\})$ and $s_1,s_2\in N(\{u,v\})$, respectively, such  that $r_1,s_1,r_2,s_2$ occur in $C$ in order. If $H$ has no tripod with feet in $N(\{u,v\})$, then, according to Robertson and Seymour \cite{RobertsonST93}, $H$ can be embedded in a disc $D$ with the vertices of $N(\{u,v\})$ on the boundary of $D$, and so the crossing number of $d_{\restriction E(H)^2}$ equals $0$. Next, we show that if there is a tripod with feet in $N(\{u,v\})$ and $H$ is mapped in a disc $D$ with the vertices of $N(\{u,v\})$ on the boundary of $D$, then we can subtract a multiple of a quad $2$-cycle from $d$ such that for the resulting $2$-cycle $d''$ the crossing number of $d''_{\restriction E(H)^2}$ equals $0$. 
In both cases,  the crossing number of $d_{\restriction E(H)^2}$ equals $0$ .  We then show that, if $H_1 = G[\delta(u)\cup \delta(v)\cup E(C)]$ and $u$ and $v$ are mapped outside the disc $D$, then the crossing number of $d_{\restriction E(H_1)^2}$ equals $0$. Using the cycle $C$, we show that we can subtract cycle-pair $2$-cycles from $d$ such that the resulting $2$-cycle $d''$ satisfies $d''_{\restriction \delta(u)\times \delta(v)} = \alpha K_{\restriction \delta(u)\times \delta(v)}$ and $d''_{\restriction \delta(v)\times \delta(u)} = \beta K_{\restriction \delta(v)\times \delta(u)}$ for some Kuratowski $2$-cycle $K$ in $H_1$ and integers $\alpha, \beta$. Using that the crossing number of $d_{\restriction E(H_1)^2}$ equals $0$, we show that $\alpha=\beta$. We can therefore subtract $\alpha K$ from $d''$ resulting in a $2$-cycle $d'$ so that $d'(f, g) = 0$ and $d'(g, f) = 0$ for every $f\in \delta(u)$ and every $g\in \delta(v)$.

%
%
An \emph{arc} of a subgraph $H$ of $G$ is a path of $H$ with distinct ends, both of degree $\geq 3$ in $H$, and with internal vertices of degree two in $H$.

Let $H_1$ and $H_2$ be Kuratowski subgraphs of $G$. A subset of edges $F$ \emph{meets} an arc if an edge of the arc belongs to $F$.
A $(\leq 3)$-separation $(G_1,G_2)$ of $G$ \emph{divides} Kuratowski subgraphs $H_1$ and $H_2$ if $E(G_1)$ meets $\leq 3$ arcs of $H_1$ and $E(G_2)$ meets $\leq 3$ arcs of $H_2$, or vice versa. A graph $G$ is Kuratowski connected if no $(\leq 3)$-separation $(G_1,G_2)$ divides Kuratowski subgraphs $H_1$ and $H_2$. In Section~\ref{sec:Kurconn},
we show the following theorems.

\begin{theorem}
Let $G$ be a Kuratowski-connected graph. Then $L(G)$ is spanned by $B(G)$ and the Kuratowski $2$-cycles.
\end{theorem}

\begin{theorem}
Let $G$ be a Kuratowski-connected graph. Then $L^{\text{skew}}(G) = B^{\text{skew}}(G)$.
\end{theorem}

In \cite{Holst2003b}, van der Holst proved the following theorem.

\begin{theorem}
Let $G$ be a Kuratowski-connected graph. If $d_H$ and $d_{H'}$ are Kuratowski $2$-cycles on Kuratowski subgraphs $H$ and $H'$ of $G$, respectively, then $d_H - d_{H'} \in B^{\text{sym}}(G)$ or $d_H + d_{H'} \in B^{\text{sym}}(G)$.
\end{theorem}

From Theorem~\ref{thm:mainKur}, we immediately obtain the following theorem.

\begin{theorem}
Let $G$ be a Kuratowski-connected graph. Then $L(G)$ is spanned by $B(G)$ and at most one Kuratowski $2$-cycle.
\end{theorem}

In Section~\ref{sec:Kurconn}, we also prove  the following theorem.
\begin{theorem}
Let $G$ be a Kuratowski-connected graph. Then $L(G)=B(G)$ if and only if $G$ is planar or $G$ does not admit a linkless embedding.
\end{theorem}

\section{Quad $2$-cycles}\label{sec:quad}

Barnett \cite{Barnett2010a} proved that there exists a $2$-cycle on $K_{3,4}$ which is not a sum of an element in $B(G)$ and Kuratowski $2$-cycles. The graph $K_{3,4}$ is a special case of a quad. On each quad there exist $2$-cycles that are not a sum of an element in $B(G)$ and Kuratowski $2$-cycles.

A \emph{quad} of a graph $G=(V,E)$ is a subgraph $K = P_1\cup P_2\cup P_3\cup Q_1\cup Q_2\cup Q_3\cup R_1\cup R_2\cup R_3$ of $G$ consisting of 
\begin{enumerate}
\item four distinct vertices $a,b,c,d$; 
\item three paths $P_1,P_2,P_3$ of $G$ between $a$ and $b$, mutually internally disjoint, each with at least one internal vertex;
\item three paths $R_1,R_2,R_3$ of $G$ between $c$ and $d$, mutually internally disjoint, each with at least one internal vertex;
\item three paths $Q_1,Q_2,Q_3$ of $G$, mutually disjoint, such that for $i=1,2,3$, $Q_i$ has ends $u_i$ and $v_i$, where $u_i\in V(P_i)\setminus \{a,b\}$, $v_i\in V(R_i)\setminus \{c,d\}$, and no vertex of $Q_i$, except for $u_i$ and $v_i$ belongs to $V(P_1\cup P_2\cup P_3\cup R_1\cup R_2\cup R_3)$;
\item for $i=1,2,3$, $V(P_i\cap (R_1\cup R_2\cup R_3)) \subseteq V(Q_i)$. 
\end{enumerate}
We allow each path $Q_i$ to consist of one vertex.
The \emph{width} of a quad is the sum of the lengths of the paths $Q_1,Q_2,Q_3$. We call the sets $\{a,b\}$ and $\{c,d\}$ the \emph{axles} of the quad.
Choose a vertex $s \in \{a,b\}$ and a vertex $t\in \{c,d\}$. For $i=1,2,3$, the ends of $Q_i$ split the paths $P_i$ and $R_i$ into two subpaths. We denote the subpath of $P_i$ containing the vertex $s$ by $P_{L,i}$ and the other by $P_{R,i}$, and, similarly,  we denote the subpath of $R_i$ containing the vertex $t$ by $R_{L,i}$ and the other by $R_{R,i}$.
Let 
\begin{equation*}
K_L = P_{L,1}\cup  P_{L,2} \cup P_{L,3}\cup Q_1\cup Q_2\cup Q_3\cup R_{L,1}\cup R_{L,2}\cup R_{L,3}
\end{equation*}
and let 
\begin{equation*}
K_R = P_{R,1}\cup  P_{R,2} \cup P_{R,3}\cup Q_1\cup Q_2\cup Q_3\cup R_{R,1}\cup R_{R,2}\cup R_{R,3}.
\end{equation*}
For $i=2,3$, let $C_i$ be the unique cycle of $K_L$ that does not contain $Q_i$ and let $D_i$ be the unique cycle of $K_R$ that does not contain $Q_i$. 
Orient the paths $P_{L,1}$ and $P_{R,1}$ from $a$ and $b$, respectively. For $i=2,3$, orient $C_i$ and $D_i$ such that $P_{L,1}$ and $P_{R,1}$ are traversed in forward direction by $C_i$ and $D_i$, respectively. 
Define  $q_{s,t} : E^2\to\mathbb{Z}$ by
\begin{equation*}
q_{s,t} = d_{C_2, D_3}  - d_{C_3, D_2}.
\end{equation*}
It is easy to verify that $q_{s,t} \in L(K)$.
We call any such $2$-cycle a \emph{quad $2$-cycle}. We call $\{s,t\}$ the \emph{left side} of the quad $2$-cycle.


The next lemma explains why in the statement of the decomposition of symmetric $2$-cycles no quad $2$-cycles appear.

\begin{lemma}\label{lem:nosymquad}
Let $q$ be a quad $2$-cycle. Then $q+T(q)$ is a sum of Kuratowski $2$-cycles.
\end{lemma}
\begin{proof}
Let $K$ be a quad supporting $q$, and let $P_1,P_2,P_3, Q_1,Q_2,Q_3,R_1,R_2,R_3$ as in the definition of a quad. Let $H$ be the Kuratowski subgraph 
\begin{equation*}
P_1\cup P_2\cup P_3\cup Q_1\cup Q_2\cup Q_3\cup R_{R,1}\cup R_{R,2}\cup R_{R,3},
\end{equation*} 
and let $d_H$ be the Kuratowski $2$-cycle with $d_H(e, f) = d_H(f, e) = 1$ if $e\in E(P_{L,1})$ and $f\in E(P_{R,3})$. Let $z=q + T(q) -d_H$. Then $z(e, \cdot) = 0$ for all edges $e$ of $\cup_{i=1}^3 P_{L,i}$. Hence we may view $z$ as a $2$-cycle on the Kuratowski subgraph 
\begin{equation*}
H' = P_{R,1}\cup P_{R,2}\cup P_{R,3}\cup Q_1\cup Q_2\cup Q_3\cup R_1\cup R_2\cup R_3. 
\end{equation*}
Any $2$-cycle on a Kuratowski subgraph is an integral multiple of a Kuratowski $2$-cycle. Hence $q+T(q) = d_{H}+d_{H'}$ for some Kuratowski $2$-cycle $d_{H'}$ on $H'$.
\end{proof}

\section{Increasing the connectivity}\label{sec:conn}

In this section, we show that it suffice to prove Theorem~\ref{thm:decomp} for $3$-connected graphs.

A \emph{separation} of a graph $G$ is a pair $(G_1,G_2)$ of subgraphs with $G_1\cup G_2 =G$ and $E(G_1)\cap E(G_2)=\emptyset$. The \emph{width} of a separation $(G_1, G_2)$ is $k := V(G_1 \cap G_2)$. We call a separation of width $k$ a \emph{$k$-separation}.
If $(G_1,G_2)$ is a separation of a graph $G$, we denote by $B(G_1, G_2)$ we denote the group generated by all $2$-cycles $d_{C, D}$ and $d_{D, C}$, with $C$ a cycle of $G_1$, $D$ a cycle of $G_2$, and $C$ and $D$ disjoint.

The support of a circulation $c$ on a graph $G$ is the set of all edges $e\in E(G)$ such that $c(e)\not=0$. 

The proofs of the following lemmas follow the proofs of similar lemmas in \cite{Holst2003b}.

\begin{lemma} \label{lem:1sep}
Let $(G_1, G_2)$ be a $(\leq 1)$-separation of $G = (V,E)$. Then
$L(G) = L(G_1) + L(G_2) + B(G_1,G_2)$.
\end{lemma}

\begin{proof}
The inclusion $L(G_1) + L(G_2) + B(G_1,G_2)\subseteq L(G)$ is
clear. To see the other inclusion, let $d \in L(G)$.
Define $d_1 = d_{\restriction E(G_1)^2}$ and $d_2 = d_{\restriction E(G_2)^2}$.
Let $d' = d-d_1-d_2$. So $d'_{\restriction E(G_1)^2} = 0$ and $d'_{\restriction E(G_2)^2} = 0$.

Order the edges of $G_1$ arbitrarily as $e_1,e_2,\ldots,e_k$ that starts with
edges in $\delta_{G_1}(u)$ for each $u\in V(G_1)\cap V(G_2)$, and the edges
of $G_2$ arbitrarily as $f_1,f_2,\ldots, f_\ell.$ Choose $i,j$ with $d'(e_i,f_j) \not= 0$ or $d'(f_j, e_i)\not=0$, and $i+j$ minimal. We assume that $d'(e_i, f_j)\not=0$, the other case is similar. Let $C$ be a cycle of $G_1$ in the support of $d'(\cdot, f_j)$ that contains $e_i$. Let $D$ be a cycle of $G_2$ in the support of $d'(e_i,\cdot)$ that contains $f_j$.
The cycles $C$ and $D$ are disjoint. This is clear
if $(G_1,G_2)$ is a $0$-separation. If $(G_1,G_2)$ is a $1$-separation and $C$ and $D$ are not disjoint,
then $C$ contains an edge $e$ incident to $u$, and hence by the ordering
chosen, $e_i$ is incident to $u$. Then $D$ does not traverse $u$, as it is in the
support of $d'(e_i,\cdot)$; a contradiction.
Now orient $C$ and $D$ in such a way that $d_{C,D}(e_i,f_j) = 1$. Replacing $d'$ by $d'-d'(e_i,f_j)d_{C,D}$ gives a reduction. Repeating this until we reach $d' = 0$, shows the
lemma.
\end{proof}

\begin{lemma} \label{lem:2sep}
Let $(G_1, G_2)$ be a $2$-separation of a $2$-connected graph $G$.
For $i = 1, 2$, let $P_i$ be a path in $G_i$ connecting both vertices in
$V(G_1)\cap V (G_2)$. 
Then $L(G) = L(G_1\cup P_2) + L(G_2\cup P_1) + B(G_1,G_2)$.
\end{lemma}

\begin{proof}
The inclusion $L(G_1\cup P_2) + L(G_2\cup P_1) + B(G_1,G_2)\subseteq L(G)$ is clear.

We now prove the converse inclusion. Let $u_1$ and $u_2$ be the vertices in $V(G_1)\cap V(G_2)$. By reorienting the edges of $P_1$ and $P_2$, we may assume that $P_1$
and $P_2$ are oriented paths. For each edge $e$ in $G_1$, let
$\phi_1(e)$ be the net inflow in $u_1$ of $d(e,\cdot)$ when restricted
to $E(G_1)$, and let $\phi_2(e)$ be the net inflow in $u_1$ of $d(\cdot, e)$ when restricted to $E(G_1)$. For edges $e,f$ in $G_1\cup P_2$ define
\begin{equation}
d_1(e,f) := \begin{cases}
d(e,f) & \text{if $e,f \in E(G_1)$,}\\
0 & \text{if $e,f \in E(P_2)$,}\\
\phi_1(e) & \text{if $e \in E(G_1)$ and $f \in E(P_2)$,}\\
\phi_2(f) & \text{if $e \in E(P_2)$ and $f \in E(G_1)$.}
\end{cases}
\end{equation}
Then $d_1 \in L(G_1\cup P_2)$. Similarly, we define
$d_2 \in L(G_2\cup P_1)$.

Let $d_3 = d - d_1 - d_2$. So $d_3(e,f) = 0$ if $e,f \in E(G_1)$ or
$e,f \in E(G_2)$. Order the edges of $G_1$ and $G_2$ as
$e_1,e_2,\ldots, e_k$ and $f_1,f_2,\ldots, f_\ell$ respectively, in such a way
that the edges in $\delta_{G_1}(u_1)$ occur first among
$e_1,e_2,\ldots, e_k$, and the edges in $\delta_{G_2}(u_2)$ occur first
among $f_1,f_2,\ldots, f_\ell$. Choose $i,j$ with $d_3(e_i,f_j) \not= 0$ or $d_3(f_j, e_i)\not=0$, and
$i+j$ minimal. We assume that $d_3(e_i, f_j)\not=0$; the case where $d_3(f_j, e_i)\not=0$ is similar. Let $C$ be a cycle in the support of
$d_3(\cdot,f_j)$ and containing $e_i$. Let $D$ be a cycle
contained in the support of $d(e_i,\cdot)$ and containing $f_j$.

Then $C$ and $D$ are cycles in $G_1$ and $G_2$, respectively,
as $d_3(e,f) = 0$ if $e,f \in E(G_1)$ or $e,f \in E(G_2)$.
Moreover, $C$ and $D$ are disjoint.  For suppose they have
a vertex in common, say $u_1$. So $d_3(e,f_j)\not= 0$ for some
$e\in \delta_{G_1}(u_1)$. Then $e_i \in \delta_{G_1}(u_1)$, by the choice
of the ordering of the edges $e_1,e_2,\ldots, e_k.$  But since the support
of $d_3(e_i, \cdot)$ contains no edges incident with $u_1$, we
arrive at a contradiction.

Choose the orientations of $C$ and $D$ such that $e_i$ and $f_j$ occur
in forward direction. Then replacing $d_3$ by
$d_3 - d_3(e_i,f_j) d_{C,D}$ gives a reduction. Repeating this
shows that $d_3 \in B(G_1,G_2)$.
\end{proof}

If $(G_1,G_2)$ is a separation of a graph $G$, we denote by $B^{\text{skew}}(G_1, G_2)$ we denote the group generated by all $2$-cycles $d_{C, D} - d_{D, C}$, with $C$ a cycle of $G_1$, $D$ a cycle of $G_2$, and $C$ and $D$ disjoint.

The proofs of the following two lemmas are similar to the proofs of Lemmas~\ref{lem:1sep} and~\ref{lem:2sep}.

\begin{lemma} \label{lem:1sepskew}
Let $(G_1, G_2)$ be a $(\leq 1)$-separation of $G = (V,E)$. Then
$L^{\text{skew}}(G) = L^{\text{skew}}(G_1) + L^{\text{skew}}(G_2) + B^{\text{skew}}(G_1,G_2)$.
\end{lemma}

\begin{lemma} \label{lem:2sepskew}
Let $(G_1, G_2)$ be a $2$-separation of a $2$-connected graph $G$.
For $i = 1, 2$, let $P_i$ be a path in $G_i$ connecting both vertices in
$V(G_1)\cap V (G_2)$. 
Then $L^{\text{skew}}(G) = L^{\text{skew}}(G_1\cup P_2) + L^{\text{skew}}(G_2\cup P_1) + B^{\text{skew}}(G_1,G_2)$.
\end{lemma}

\section{Contracting an edge}

Let $e = uv$ be an edge of a graph $G = (V,E)$ and
let $d\in L(G)$. If $d(f, g) = 0$ and $d(g, f)=0$ for all $f\in \delta(u)$ and all $g\in \delta(v)$, we define $d/e = d_{\restriction (E\setminus \{e\})^2}$. 
It is easy to see that $d/e\in L(G/e)$. If we show that there exist cycle-pair $2$-cycles $d_{C_i, D_i}$, $i=1,\ldots,k$, quad $2$-cycles $q_i$, $i=1,\ldots,\ell$, and Kuratowski $2$-cycles $d_{H_i}$, $i=1,\ldots, m$ such that $d' = d - \sum_{i=1}^k d_{C_i, D_i} - \sum_{i=1}^{\ell} q_i - \sum_{i=1}^m d_{H_i}$ satisfies $d'(f,g) = 0$ and $d'(g, f)$ for all $f\in \delta(u)$ and all $g\in \delta(v)$, then $d'/e$ is a $2$-cycle on $G/e$. By induction, $L(G/e)$ is spanned by $B(G/e)$, the Kuratowski $2$-cycles, and the quad $2$-cycles of $G/e$. In this section, we study the $2$-cycles $d$ on $G$ if $d/e$ is a cycle-pair $2$-cycle, a Kuratowski $2$-cycle, or a quad $2$-cycle.

\begin{lemma} \label{lem:uncontract}
Let $G=(V,E)$ be a graph and $e$ be an edge of $G$. Then, for any $d'\in L(G/e)$, there exists a unique $d \in L(G)$ such that $d/e = d'$. 
\end{lemma}
\begin{proof}
Let $d'\in L(G/e)$. Let $u$ and $v$ be the ends of $e$. Define $d : E^2\to \mathbb{Z}$ by
\begin{equation*}
d(f, g) = \begin{cases}
d'(f, g) & \text{if $f\not=e$, $g\not=e$,}\\
-[u, e] \sum_{h\in E\setminus \{e\}}[u, h]d'(h, g) &  \text{if $f=e$, $g\not=e$,}\\
-[u, e] \sum_{h\in E\setminus \{e\}}[u, h]d'(f, h) &  \text{if $f\not=e$, $g=e$,}\\
0 & \text{if $f=e$, $g=e$.}
\end{cases}
\end{equation*}
Then $d(f,\cdot)$ and $d(\cdot, f)$ are circulations for each edge $f$ of $G$.

It remains to show that $d(f, g) = 0$ if $g$ and $f$ are edges that have a common vertex. 
If $f,g \in E\setminus \{e\}$ and $f$ and $g$ are adjacent, then $d(f, g) = d'(f, g) = 0$.
If $f\in E\setminus \{e\}$ and $f$ and $e$ are adjacent, then 
\begin{equation*}
d(f, e)  = -[u, e]\sum_{h\in E \setminus \{e\}}[u, h] d'(f, h) = 0.
\end{equation*}
In the same way, $d(e, f) = 0$ for every edge $f$ that share a vertex with $e$. Furthermore,
$d(e,e) = 0$.
Hence $d(f, g) = 0$ if $g$ and $f$ have a common vertex. Hence $d\in L(G)$.  It is clear that $d$ is unique.
\end{proof}

The following lemma is easy to verify.
\begin{lemma}\label{lem:contrKur}
Let $G=(V,E)$ be a graph, let $e$ be an edge of $G$, and let $d \in L(G)$. Then
\begin{enumerate}[(i)]
\item\label{item:disjcircuits} if $d/e = d_{C', D'}$ for disjoint oriented cycles $C'$ and $D'$ of $G/e$, then $d = d_{C, D}$ for disjoint oriented cycles $C$ and $D$ of $G$;
\item if $d/e$ is a Kuratowski $2$-cycle on some $K_{3,3}$-subdivision $H'$ in $G/e$, then $d$ is a Kuratowski $2$-cycle on some $K_{3,3}$-subdivision
$H$ in $G$, with $H/e = H'$; 
\item if $d/e$ is a Kuratowski $2$-cycle on  $H'$ for some $K_5$-subdivision $H'$ in $G/e$, then $d = d_{H} + \alpha (d_{C,D} + d_{D,C}))$ for some $\alpha \in \{0,1\}$, some
disjoint oriented cycles $C$ and $D$ of $G$,  and a Kuratowski $2$-cycle $d_{H}$ on
some $K_5$- or $K_{3,3}$-subdivision  $H$ in $G$, contained in a subgraph $H''$ of
$G$ with $H''/e = H'$;
\end{enumerate}
\end{lemma}

For quad $2$-cycles on a graph, we have the following lemma.
\begin{lemma}\label{lem:contrquad}
Let $G=(V,E)$ be a graph, let $e$ be an edge of $G$, and let $q\in L(G)$. Then, if $q/e$ is a quad $2$-cycle on some quad $K'$ in $G/e$, then either $q$ is a quad $2$-cycle on a quad $K$ in $G$ or $q\in B(G)$.
\end{lemma}
\begin{proof}
Let $v_e$ be the vertex in $G/e$ obtained by contracting $e$. If $v_e$ has degree $(\leq 3)$ in $G/e$, then $G$ is a quad and $q$ is a quad $2$-cycle on $G$.

Suppose next that $v_e$ has degree four in $G/e$.  Then, either $G$ is a quad and $q$ is a quad $2$-cycle on $G$, or $G$ can be written as 
$P_1\cup P_2\cup P_3\cup Q_1\cup Q_2\cup Q_3\cup R_1\cup R_2\cup R_3$, consisting of
\begin{enumerate}[(1)]
\item four distinct vertices $a,b,c,d$;
\item three paths $P_1,P_2,P_3$ between $a$ and $b$, mutually internally disjoint, and each with at least one internal vertex;
\item internally disjoint paths $R_1,R_2,R_3$ between $c$ and $d$, mutually internally disjoint, and each with at least one internal vertex;
\item two paths $Q_2,Q_3$, mutually disjoint, such that for $i=2,3$, $Q_i$ has ends $u_i$ and $v_i$, where $u_i\in V(P_i) \setminus \{a,b\}$, $v_i\in V(R_i)\setminus \{c,d\}$, and no vertex of $Q_i$, except for $u_i$ and $v_i$ belong to $V(P_1\cup P_2\cup P_3\cup R_1\cup R_2\cup R_3)$;
\item for $i=2,3$, $V(P_i\cap (R_1\cup R_2\cup R_3))\subseteq V(Q_i)$.
\item $P_1$ is disjoint from $R_2\cup R_3$, and $P_1\cap R_1$ is a path of length one and $\{a,b,c,d\}\cap V(P_1\cap R_1) = \emptyset$.
\end{enumerate}
Let $C_2, C_3, D_2, D_3$ be the oriented cycles as in the definition of the quad $2$-cycles $q_{s,t}$ on $G/e$. Since 
$C_2$ and $D_3$ are disjoint and $C_3$ and $D_2$ are disjoint in $G$, $q\in B(G)$.
\end{proof}


\section{$2$-cycles of a graph at an edge}

Let $e=uv$ be an edge of a graph $G$. For any $d\in L(G)$, we define 
\begin{equation*}
P_{u,v}(d) = d_{\restriction \delta(u)\times \delta(v)}.
\end{equation*} 
If $d$ is a symmetric $2$-cycle, then $P_{u,v}(d) = P_{v,u}(d)$, and if $d$ is a skew-symmetric $2$-cycles, then $P_{u,v}(d) = -P_{v,u}(d)$.

Let $d\in L(G)$. We want to show that that there exist cycle-pair $2$-cycles $d_{C_i, D_i}$, $i=1,\ldots,k$, quad $2$-cycles $q_i$, $i=1,\ldots,\ell$, and Kuratowski $2$-cycles $d_{H_i}$, $i=1,\ldots, m$ such that $d' = d - \sum_{i=1}^k d_{C_i, D_i} - \sum_{i=1}^{\ell} q_i - \sum_{i=1}^m d_{H_i}$ satisfies $d'(f,g) = d'(g, f) = 0$ for all $f\in \delta(u)$ and all $g\in \delta(v)$. This can be formulated as $P_{u,v}(d') = 0$ and $P_{v,u}(d') = 0$.

If $f_1, f_2 \in \delta(u)\setminus \{e\}$ are distinct edges and $g_1, g_2 \in \delta(v)\setminus \{e\}$ are distinct edges, we denote by $[f_1,f_2; g_1,g_2]_{\delta(u)\times\delta(v)} = [f_1,f_2; g_1,g_2]$ the function $c : \delta(u)\times \delta(v)\to \mathbb{Z}$ with $c(f_1,g_1) = 1$, for $i=1,2$, 
\begin{align*}
[u,f_1]c(f_1,g_i) + [u, f_2]c(f_2, g_i) & = 0,\\
[v, g_1]c(f_i, g_1) + [v, g_2] c(f_i, g_2) & = 0,
\end{align*}
and $c(f,g)=0$ if $f\not\in \{f_1,f_2\}$ or $g\not\in \{g_1, g_2\}$. We will refrain from using the notation $[f_1,f_2; g_1,g_2]_{\delta(u)\times\delta(v)}$ and use only the notation $[f_1,f_2; g_1,g_2]$ as this will cause no confusion. We call the function $[f_1,f_2; g_1,g_2]$ a \emph{$4$-cross at $u,v$}.

If $f_1= uz_1, f_2 = uz_2, f_3 = uz_3, g_1 = vz_1, g_2 = vz_2, g_3 = vz_3$ are distinct edges, we denote by $[f_1, f_2, f_3; g_1, g_2, g_3]_{\delta(u)\times\delta(v)} = [f_1, f_2, f_3; g_1, g_2, g_3]$ the function $c : \delta(u)\times \delta(v)\to \mathbb{Z}$ with $c(f_1,g_2) = 1$,
\begin{align*}
[v, g_2]c(f_1, g_2) + [v, g_3]c(f_1, g_3) & = 0,\\
[u, f_1]c(f_1, g_3) + [u, f_2] c(f_2, g_3) & = 0,\\
[v, g_3] c(f_2, g_3) + [v, g_1] c(f_2, g_1) &= 0,\\
[u, f_2] c(f_2, g_1) + [u, f_3] c(f_3, g_1) &= 0,\\ 
[v, g_1] c(f_3, g_1) + [v, g_2] c(f_3, g_2) &= 0,\\
[u, f_3] c(f_3, g_2) + [u, f_1] c(f_1, g_2) &= 0,
\end{align*}
$c(f,g)=0$ if $f\not\in \{f_1,f_2,f_3\}$ or $g\not\in \{g_1, g_2, g_3\}$, and $c(f_1,g_1) = c(f_2, g_2) = c(f_3, g_3) = 0$. We call the function $[f_1,f_2,f_3; g_1, g_2,g_3]$ a \emph{$6$-cross at $u,v$}.

If $C$ and $D$ are disjoint cycles of $G$ with $f_1,f_2\in E(C)$ and $g_1,g_2\in E(D)$, then there exist orientations of $C$ and $D$ such that $P_{u,v}(d_{C,D}) = [f_1,f_2; g_1,g_2]$. Observe that $P_{v,u}(d_{C,D}) = 0$.
Let $K = P_1\cup P_2\cup P_3\cup Q_1\cup Q_2\cup Q_3\cup R_1\cup R_2\cup R_3$ be a quad on $G$, where we assume that $P_1,P_2,P_3, Q_1,Q_2,Q_3, R_1, R_2, R_3$ are as in the definition of a quad. Let $\{s,t\}$ be the left side of the quad. Suppose $Q_1$ has length $1$ and $R_{R,1}$ consists of the edge $e=uv$ only, where $u\in V(Q_1)$. Let $g_1,g_2\in \delta(v)\setminus\{e\}$ be distinct edges and let $f_1,f_2\in \delta(u)$ with $f_1\in E(P_{L,1})$ and $f_2\in E(R_{L,1})$. If $q$ is a quad $2$-cycle on $K$ with left side $\{s,t\}$ and $q(f_1,g_1) = 1$, then $P_{u,v}(q) = [f_1,f_2; g_1,g_2]$ and $P_{v,u}(q) = 0$. Let $H$ be a $K_{3,3}$-subdivision in $G$, where one arc of $H$ consist of the edge $e=uv$ only. Let $f_1,f_2$ and $g_1,g_2$ be the edges incident with $u$ and $v$, respectively, that are unequal to $e$. If $d_H$ is a Kuratowski $2$-cycle on $H$ with $d_H(f_1,g_1) = 1$, then $P_{u,v}(d_H) = [f_1,f_2; g_1,g_2]$ and $P_{v,u}(d_H) = [g_1,g_2; f_1,f_2]$. 
Let $H$ be a $K_5$-subdivision of $G$, where one arc of $H$ consists of the edge $e=uv$ only. Let $f_1,f_2,f_3$ and $g_1,g_2,g_3$ be the edges incident with $u$ and $v$, respectively, that are unequal to $e$ such that $f_i$ and $g_i$ are adjacent for $i=1,2,3$. If $d_H$ is a Kuratowski $2$-cycle on $H$ with $d_H(f_1,g_1) = 1$, then $P_{u,v}(d_H) = [f_1,f_2,f_3; g_1,g_2,g_3]$ and $P_{v,u}(d_H) = [g_1,g_2,g_3; f_1,f_2,f_3]$.

If $e=uv$ is an edge, we denote by $B_{u,v}(G)$ the subgroup of $B(G)$ generated by the set of all cycle-pair $2$-cycles $d_{C, D}$ with $u\in V(C)$ and $v\in V(D)$. 

\begin{lemma}\label{lem:Kuratowskiconnected2}
Let $G$ be a graph, let $e=uv$ be an edge, and let $H := G-\{u,v\}$. Let $z_1,z_2,z_3\in N_H(u)\cap N_H(v)$ be distinct vertices and let $z\in N(\{u,v\})\setminus \{z_1,z_2,z_3\}$. If there exists a cycle $C$ in $H$ such that $z_1,z_2,z_3,z$ occur in this order on $C$, then 
$\mat{uz}{uz_2}{vz_1}{vz_3} - [uz_1, uz_2, uz_3; vz_1, vz_2, vz_3]\in B_{u,v}(G).$
\end{lemma}

\begin{proof}
By symmetry, we may assume that $z\in N_H(u)$. Let $P_1, Q_1$ be disjoint paths in $C$, with $P_1$ connecting $z$ to $z_3$, and $Q_1$ connecting $z_1$ to $z_2$. Let $P_2, Q_2$ be disjoint paths in $C$, with $P_2$ connecting $z$ and $z_1$, and $Q_2$ connecting $z_2$ and $z_3$. Let $C_1=uzP_1z_3u$, let $D_1=vz_1Q_1z_2v$, let $C_2=uzP_2z_1u$, and $D_2=vz_2Q_2z_3v$.
Orient $C_1, D_1$ and $C_2, D_2$ such that $P_{u,v}(d_{C_1, D_1}) = \mat{uz}{uz_3}{vz_1}{vz_2}$
and $P_{u,v}(d_{C_2, D_2}) = \mat{uz}{uz_1}{vz_2}{vz_3}.$ Then
\begin{equation*}
P_{u,v}(d_{C_1, D_1} + d_{C_2, D_2}) = \mat{uz}{uz_2}{vz_1}{vz_3} - [uz_1, uz_2, uz_3; vz_1, vz_2, vz_3].
\end{equation*}
Hence $\mat{uz}{uz_2}{vz_1}{vz_3} - [uz_1, uz_2, uz_3; vz_1, vz_2, vz_3]\in B_{u,v}(G)$.
\end{proof}

The proof of the following lemma is easy.
\begin{lemma}\label{lem:cyclebuilt2}
Let $H$ be a $2$-connected graph and let $u_1,u_2,v_1,v_2$ be distinct vertices of $H$.
Then there exists a cycle $C$ of $H$ containing the vertices $u_1,u_1, v_1, v_2$.
\end{lemma}

\begin{lemma}\label{lem:4cross6cross}
Let $e=uv$ be an edge of a $3$-connected graph $G$ such that $G/e$ is $3$-connected. Let $H=G-\{u,v\}$, and let $d$ be a $2$-cycle on $G$. If $N(\{u,v\})$ contains at least four vertices, then $P_{u,v}(d) =\sum \alpha_i F_i$, where each $F_i$ is a $4$-cross at $u,v$ and each $\alpha_i$ is an integer, and if $N(\{u,v\})$ contains exactly three vertices, then $P_{u,v}(d) = \alpha K$ for some $6$-cross $K$ at $u,v$ and some integer $\alpha$.
\end{lemma}
\begin{proof}
Order the edges in $\delta (u) \setminus \{ e \}$ as
$f_1,\ldots, f_k$ in such a way that we start with the edges that connect
$u$ to a neighbor of $v$. Similarly, we order the edges in
$\delta (v)\setminus \{ e \}$ as $g_1, \ldots, g_{\ell}$ in such a way that we
start with the edges that connect $v$ to a neighbor of $u$. 
For $t=1,\ldots,k$, let $f_t$ have ends $u$ and $u_t$. For $t=1,\ldots,\ell$, let $g_t$ have ends $v$ and $v_t$. 

Choose $i$ and $j$ with $d(f_i, g_j)\not= 0$ and $i + j$ minimal. 
Let $f_{i'}$ be an edge in the support of $d(\cdot,g_j)$
that is unequal to $f_i$, and let $g_{j'}$ be an
edge in the support of $d(f_i,\cdot)$ that is unequal to $g_j$.
These edges exist since $d(\cdot,g_j)$ and $d(f_i,\cdot)$ are circulations.
Since $f_i$ and $g_j$ are nonadjacent, we know that $u_i\not=v_j$.
Similarly, we know that $u_i\not=v_{j'}$ and $v_j\not=u_{i'}$. We consider
now several cases.

In the first case we assume $u_{i'}\not=v_{j'}$. Then $\mat{f_i}{f_{i'}}{g_j}{g_{j'}}$ is a $4$-cross at $u,v$. Replacing $d$ by $d- d(f_i,g_j)\mat{f_i}{f_{i'}}{g_j}{g_{j'}}$ gives a reduction using $i'>i,j'>j$. 

In the second case we assume that $u_{i'}=v_{j'}$. By the orderings of the edges $f_1,f_2,\ldots, f_k$ and $g_1,g_2,\ldots, g_{\ell}$ and by the minimality of $i+j$, $u_i$ is adjacent to $v$, and $v_j$ is adjacent to $u$. So each of $u_i,v_j$ and $u_{i'} (= v_{j'})$ is
adjacent to $u$ and $v$. Then $[f_i, f_{i'}, f_j; g_{j}, g_{j'}, g_i]$ is a $6$-cross at $u,v$.  Replacing $d$ by $d- d(f_i,g_j)[f_i, f_{i'}, f_j; g_{j}, g_{j'}, g_i]$ gives a reduction using $i'>i,j'>j$. 

Since $G/e$ is $3$-connected, $G-\{u,v\}$ is $2$-connected. If $N(\{u,v\})$ has at least four vertices, let $z$ be a vertex in $N(\{u,v\})\setminus \{u_i,v_j,u_{i'}\}$. By Lemma~\ref{lem:cyclebuilt2}, there exists a cycle $C$ in $G-\{u,v\}$ containing the vertices $u_i, v_j, u_{i'}, z$.  By Lemma~\ref{lem:Kuratowskiconnected2}, $[f_i, f_{i'}, f_j; g_{j}, g_{j'}, g_i]$ is a sum of a $4$-cross at $u,v$ and an element of $P_{u,v}(B_{u,v}(G))$.
\end{proof}

\section{Obtaining a cycle}

Let $e=uv$ be an edge of a $3$-connected graph $G$ such that $G/e$ is $3$-connected, and let $d$ be a $2$-cycle on $G$. Suppose that $N(\{u,v\})$ has at least four vertices. In this section, we show that if there are no cycle-pair $2$-cycles $d_{C_i, D_i}$, $i=1,\ldots,k$, and quad $2$-cycle $q_i$, $i=1,\ldots,\ell$, such that $d' = d-\sum_{i=1}^k d_{C_i, D_i} - \sum_{i=1}^{\ell} q_i$ satisfies $P_{u,v}(d') = 0$ and $P_{v,u}(d') = 0$, then there exists a cycle $C$ in $G-\{u,v\}$ such that $N(\{u,v\})\subseteq V(C)$.

Let $C$ be an oriented cycle of a graph $G$. If $u,v\in V(C)$, we denote by $C[u,v]$ the path in $C$ when traversing $C$ from $u$ to $v$ in forward direction. By $C(u,v)$ we denote $C[u,v]-\{u,v\}$.

For a walk $W_1$ from vertex $s_1$ to vertex $s_2$ and a walk $W_2$ from $s_2$ to $s_3$, we denote by $W_1W_2$ the walk from $s_1$ to $s_3$ by concatenating $W_1$ and $W_2$. We view an edge $f=st$ as a walk from $s$ to $t$.

\begin{lemma}\label{lem:cyclesum}
Let $G=(V,E)$ be a graph, let $e=uv$ be an edge of $G$, let $H = G-\{u,v\}$. Let $C$ be a cycle of $H$ and let $u_1,v_1,u_2,v_2$ be distinct vertices on $C$ in this order, such that $u_1,u_2\in N(u)$, and $v_1,v_2\in N(v)$. If there exists a component $B$ of $H-V(C)$ with a vertex in $N(u)$, such that $N_H(B)$ has a vertex in $C(v_2,v_1)$ and a vertex in $C(v_1,v_2)$, then there exist cycle-pair $2$-cycles $d_{C_1,D_1},d_{C_2, D_2}$ such that $P_{u,v}(d_{C_1, D_1} + d_{C_2, D_2}) = \mat{uu_1}{uu_2}{vv_1}{vv_2}$. 
\end{lemma}
\begin{proof}
Let $u'$ be a vertex of $B$ adjacent with $u$. Let $P_1$ and $P_2$ be paths in $G[N_H(B)\cup V(B)]$ connecting $u_1$ and $u'$, and $u_2$ and $u'$, respectively, with internal vertices in $B$. Let $C_1=uu_1P_1u'u$, let $D_1=vv_1C[v_1,v_2]v_2v$, let $C_2=uu_2P_2u'u$, and let $D_2=vv_2C[v_2,v_1]v_1v$. Let $d_{C_1, D_1}(u_1u, v_1v) = 1$ and $d_{C_2, D_2}(uu_2, vv_2) = 1$. Then $P_{u,v}(d_{C_1,D_1}+d_{C_2, D_2}) = \mat{uu_1}{uu_2}{vv_1}{vv_2}$.
\end{proof}

Let $u_1,u_2,v_1,v_2$ be vertices in $V$. A \emph{$(u_1u_2; v_1v_2)$-linkage} in $G$ is a pair of disjoint paths $(P_1, P_2)$ in $G$ such that $P_1$ connects $u_1$ and $u_2$, and $P_2$ connected $v_1$ and $v_2$.

\begin{lemma}\label{lem:cyclebuilt}
Let $H$ be a $2$-connected graph and let $u_1,u_2,v_1,v_2$ be distinct vertices of $H$.
If there exist no $(u_1u_2; v_1v_2)$-linkage in $H$, then there exists a cycle $C$ of $H$ such that $u_1,v_1,u_2,v_2$ occur in this order on $C$.
\end{lemma}
\begin{proof}
Since $H$ is $2$-connected and there is no $(u_1u_2; v_1v_2)$-linkage in $H$, there exist $(u_1v_1; u_2v_2)$- and $(u_1v_2; u_2v_1)$-linkages in $H$. Hence there exist a cycle $C$ and paths $P_1, Q_1, P_2, Q_2$ connecting $C$ and $u_1, v_1, u_2, v_2$, respectively, with the ends of $P_1, Q_1, P_2, Q_2$ on $C$ in the specified order.
Choose $C$, $P_1, Q_1, P_2,$ and $Q_2$ such that the sum of the lengths of $P_1, Q_1, P_2,$ and $Q_2$ is minimal. Let $z$ be the end of $P_1$ on $C$. Suppose $P_1$ has length $>0$. Since $H$ is $2$-connected, there exists a path $P_1'$ from $u_1$ to $Q_1\cup P_2\cup Q_2\cup C$ disjoint from $z$. Let $z_1,z_2$ be the ends of $Q_1,Q_2$ on $C$, respectively, and let $R$ be the path in $C$ between, and including $z_1$ and $z_2$, containing $z$. Then $P_1'$ ends on $Q_1\cup Q_2\cup R$. Then we can find a cycle $C'$ and paths $P_1, Q_1, P_2, Q_2$ connecting $C$ and $u_1, v_1, u_2, v_2$, respectively, such that the sum of the lengths of $P_1, Q_1, P_2,$ and $Q_2$ is smaller, a contradiction. In the same way, $Q_1, Q_2,$ and $P_2$ have zero length.
\end{proof}

Let $C$ be an oriented cycle of a graph $G$. If $A, B\subseteq V(C)$, we say that $A$ and $B$ \emph{cross in $C$} if there exist distinct vertices $v_1,v_2,v_3,v_4\in V(C)$, with $v_1,v_3\in A$ and $v_2,v_4\in B$ such that $v_1,v_2,v_3,v_4$ occur in this order in $C$.

\begin{lemma}\label{lem:tocycle}
Let $G$ be a $3$-connected graph and let $e=uv$ be an edge of $G$ such that $G/e$ is $3$-connected. Let $f_1,f_2\in \delta(u)\setminus \{e\}$ and let $g_1, g_2\in \delta(v)\setminus \{e\}$ such that the ends of $f_1,f_2,g_1,g_2$ in $N(\{u,v\})$ are distinct. Then at least one of the following holds:
\begin{enumerate}
\item there exists a $d\in B_{u,v}(G)$ such that $P_{u,v}(d) = \mat{f_1}{f_2}{g_1}{g_2}$;
\item there exist a quad $2$-cycle $q$ on $G$ and a $d\in B_{u,v}(G)$ such that $P_{u,v}(q+d) = \mat{f_1}{f_2}{g_1}{g_2}$;
\item there exists an oriented cycle $C$ in $H := G-\{u,v\}$ containing all vertices of $N(\{u,v\})$, and  $N_H(u)$ and $N_H(v)$ cross in $C$.
\end{enumerate}
\end{lemma}
\begin{proof}
Let $u_1, u_2$ be the ends of $f_1,f_2$ in $N_H(u)$, respectively, and let $v_1,v_1$ be the ends of $g_1,g_2$ in $N_H(v)$, respectively. 

First suppose that there exist disjoint paths $P$ and $Q$ in $H$ with $P$ connecting $u_1$ and $u_2$, and $Q$ connecting $v_1$ and $v_2$. Then $C_1 := f_1Pf_2$ and $D_1 := g_1Qg_2$ are disjoint cycles. Orient $C_1$
and $D_1$ such that $d_{C_1, D_1}(f_1, g_1) = 1$. Then $P_{u,v}(d_{C_1, D_1}) = [f_1,f_2; g_1,g_2].$

Next suppose that such paths do not exists. Since
$G/e$ is $3$-connected, $H$ is $2$-connected. Hence, by Lemma~\ref{lem:cyclebuilt}, there exists a cycle $C$ in $H$ with $u_1, v_1, u_2, v_2$ in this order on $C$.  Choose such a cycle $C$ in $H$ containing the most vertices of $N(\{u,v\})$. Orient $C$ arbitrarily.

We now show that we may assume that $N_H(u) \subseteq V(C)$.
If there exists a component $B$ of $H-V(C)$ with a vertex in $N_H(u)$ such that $B$ has a neighbor in $C(v_2, v_1)$ and a neighbor in $C(v_1, v_2)$, then, by Lemma~\ref{lem:cyclesum}, there exist cycle-pair $2$-cycles $d_{C_1,D_1},d_{C_2, D_2}$ such that $$P_{u,v}(d_{C_1,D_1} + d_{C_2, D_2}) = [f_1,f_2; g_1,g_2].$$

We may therefore assume that $N_H(B)\subseteq C[v_2,v_1]$ or $N_H(B)\subseteq C[v_1,v_2]$ for any component $B$ of $H-V(C)$ containing a vertex in $N(u)$. Let $B$ be such a component of $H-V(C)$.  By symmetry, we may assume that $N_H(B)\subseteq C[v_2,v_1]$. For $i=1,2$, $r_i$ be the nearest vertex of $N_H(B)$ to $v_i$. Let $u_0\in V(B)\cap N(u)$.

If no vertex of $C(r_2,r_1)$ belongs to $N(\{u,v\})$, then we can find a cycle with $u_1, v_1, u_2, v_2$ occurring in this order containing more vertices of $N(\{u,v\})$; a contradiction. Hence we may assume that $C(r_2, r_1)$ contains a vertex of $N(\{u,v\})$.

Suppose now that $C(r_2, r_1)$ contains a vertex $u'\in N_H(u)$.  For $i=1,2$, let $P_i$ be a path in the subgraph $G[V(B)\cup \{r_1,r_2\}]$ connecting $u_0$ and $r_i$. As $H$ is $2$-connected, we may assume that $P_1$ and $P_2$ are internally disjoint. Let 
$$K := C\cup P_1\cup P_2\cup G[\{u_0u, u'u, f_2, e, g_1, g_2\}].$$ Then $K$ is a quad subgraph. Let $q$ be a quad $2$-cycle on $K$ with left side  $\{u_2, u'\}$ and with $q(u'u, g_1) = 1$. Then $P_{u,v}(q)= \mat{uu'}{f_2}{g_1}{g_2}.$ If $u'$ is a vertex of $C(v_2, u_1)$, let $C_1 = uu'C[u',u_1]u_1u$ and let $D_1 = g_1C[v_1, v_2]g_2$. Orient $C_1$ and $D_1$ such that $d_{C_1,D_1}(uu', g_1) = -1$. Then $$P_{u,v}(d_{C_1, D_1} + q_{u_2,u'}) = \mat{f_1}{f_2}{g_1}{g_2}.$$ The case where $u'$ is a vertex of $C(u_1, v_1)$ is similar.

Suppose next that $C(r_2,r_1)$ contains no vertex of $N_H(u)$, but a vertex $v'\in N_H(v)$. If $v'$ is a vertex of $C(u_1, v_1)$, let $C_1=f_2C[u_2,u_1]f_1$ and let $D_1=vv'C[v', v_1]g_1$. Orient $C_1$ and $D_1$ such that $d_{C_1, D_1}(f_1, g_1) = 1$. Then $P_{u,v}(d_{C_1, D_1}) = \mat{f_1}{f_2}{g_1}{v'v}.$ By Lemma~\ref{lem:cyclesum}, there exist cycle-pair $2$-cycles $d_{C_2, D_2}$ and $d_{C_3, D_3}$ such that $$P_{u,v}(d_{C_2, D_2} + d_{C_3, D_3}) = \mat{f_1}{f_2}{v'v}{g_2}.$$ Then $$P_{u,v}(d_{C_1, D_1} + d_{C_2, D_2} + d_{C_3, D_3}) = [f_1,f_2; g_1,g_2].$$ The case where $v'$ is a vertex of $C(v_2, u_1)$ is similar.
Therefore, we may assume that no component of $H-V(C)$ has a vertex of $N_H(u)$. In the same way, we may assume that no component of $H-V(C)$ has a vertex of $N_H(v)$.
Therefore, $C$ contains all vertices of $N(\{u,v\})$. Furthermore,  $N_H(u)$ and $N_H(v)$ cross in $C$.
\end{proof}

From the previous lemma and Lemmas~\ref{lem:Kuratowskiconnected2} and~\ref{lem:4cross6cross}, we obtain the following lemma.
\begin{lemma}\label{lem:tocycleb}
Let $G$ be a $3$-connected graph and let $e=uv$ be an edge of $G$ such that $G/e$ is $3$-connected. Suppose $N(\{u,v\})$ has at least four vertices. Then one of the following holds:
\begin{enumerate}
\item for any $2$-cycle on $G$, there exist a $d_1\in B_{u,v}(G)$, a $d_2\in B_{v,u}(G)$, and quad $2$-cycles $q_i$, $i=1,\ldots,k$, such that $d' = d-d_1-d_2-\sum_{i=1}^k q_i$ satisfies $P_{u,v}(d') = 0$ and $P_{v,u}(d') = 0$; or
\item there exists an oriented cycle $C$ in $H:=G-\{u,v\}$ containing all vertices of $N(\{u,v\})$, and  $N_H(u)$ and $N_H(v)$ cross in $C$.
\end{enumerate}
\end{lemma}

\section{Using the cycle}

Let $e=uv$ be an edge of a $3$-connected graph such that $G/e$ is $3$-connected, and let $d$ be a $2$-cycle on $G$. In the previous section, we proved that if there are no $d_1\in B_{u,v}(G)$, $d_2\in B_{v,u}(G)$, and quad $2$-cycles $q_i$, $i=1,\ldots,k$, such that $d' = d-d_1-d_2-\sum_{i=1}^k q_i$ satisfies $P_{u,v}(d') = 0$ and $P_{v,u}(d') = 0$, then there exists an oriented cycle $C$ in $H:=G-\{u,v\}$ containing all vertices of $N(\{u,v\})$, and  $N_H(u)$ and $N_H(v)$ cross in $C$. Let $K$ be a subdivision of $K_{3,3}$ in $G[\delta(u)\cup \delta(v)\cup E(C)]$ in which $e$ is an arc of $K$ and let $d_K$ be a Kuratowski $2$-cycle on $K$. In this section, we show, using the cycle $C$, that for any $2$-cycle $d$, there exist a $d'\in B_{u,v}(H_1)$, a $d''\in B_{v,u}(H_1)$, and integers $\alpha$ and $\beta$ such that 
\begin{align*}
P_{u,v}(d-d'-d'') &= \alpha P_{u,v}(d_K),\\ 
P_{v,u}(d-d'-d'') &= \beta P_{v,u}(d_K). 
\end{align*}

\begin{lemma}\label{lem:uniqueness}
Let $e=uv$ be an edge of a graph $G$ such that $G-\{u, v\}$ has a cycle $C$ containing all vertices in $N(\{u,v\})$. 
Let $f_1=uu_1, g_1=vv_1, f_2=uu_2, g_2=vv_2$ be distinct edges and let $\overline{f}_1=u\overline{u}_1, \overline{g}_1=v\overline{v}_1, \overline{f}_2=u\overline{u}_2, \overline{g}_2=v\overline{v}_2$ be distinct edges such that the sequences $u_1,v_1,u_2,v_2$ and $\overline{u}_1,\overline{v}_1,\overline{u}_2,\overline{v}_2$ occur in this order on $C$.
Then 
\begin{equation*}
[f_1,f_2; g_1,g_2] - [\overline{f}_1, \overline{f}_2; \overline{g}_1,\overline{g}_2] \in P_{u,v}(B_{u,v}(G)).
\end{equation*}
\end{lemma}
\begin{proof}
Suppose to the contrary that the statement of the lemma is false.
Choose the sequences  $f_1,g_1,f_2,g_2$ and $\overline{f}_1,\overline{g}_1,\overline{f}_2,\overline{g}_2$ such that 
\begin{equation*}
[f_1,f_2; g_1,g_2] - [\overline{f}_1, \overline{f}_2; \overline{g}_1,\overline{g}_2]\not\in P_{u,v}(B_{u,v}(G)),
\end{equation*}
 and $|\{f_1,f_2,g_1,g_2\}\cap \{\overline{f}_1,\overline{f}_2,\overline{g}_1,\overline{g}_2\}|$ is maximal.

First suppose $\{u_1,u_2\} = \{\overline{v}_1,\overline{v}_2\}$ and $\{v_1,v_2\}= \{\overline{u}_1,\overline{u}_2\}$. We assume that $u_1 = \overline{v}_1$, $u_2 = \overline{v}_2$, $v_1 = \overline{u}_1$, and $v_2 = \overline{u}_2$; the other cases are similar. The cycle $C$ contains a $(u_1\overline{u}_1; \overline{v}_2v_2)$-linkage $(P_1, Q_1)$, a $(\overline{u}_1u_2; v_2\overline{v}_1)$-linkage $(P_2, Q_2)$, a $(u_2\overline{u}_2; \overline{v}_1v_1)$-linkage $(P_3, Q_3)$, and a $(u_1\overline{u}_2, v_1\overline{v}_2)$-linkage $(P_4, Q_4)$. Let $C_1 = f_1P_1\overline{f}_1$, $D_1 = \overline{g}_2Q_1g_2$, $C_2 = \overline{f}_1P_2f_2$, $D_2 = \overline{g}_1Q_2g_2$, $C_3 = f_2P_3\overline{f}_2$, $D_3 = \overline{g}_1Q_3g_1$, $C_4 = f_1P_4\overline{f}_2$, and $D_4 = g_1Q_4\overline{g}_2$. Orient the cycles $C_1,\ldots, C_4$ and $D_1, \ldots, D_4$ such that $P_{u,v}(d_{C_1, D_1}) = \mat{\overline{f}_1}{f_1}{g_2}{\overline{g}_2}$,
$P_{u,v}(d_{C_2, D_2}) = \mat{f_2}{\overline{f}_1}{\overline{g}_1}{g_2}$,
$P_{u,v}(d_{C_3, D_3}) = \mat{\overline{f}_2}{f_2}{g_1}{\overline{g}_1}$,
$P_{u,v}(d_{C_4, D_4}) = \mat{f_1}{\overline{f}_2}{\overline{g}_2}{g_1}.$
Then
\begin{equation*}
[f_1,f_2; g_1,g_2] - [\overline{f}_1, \overline{f}_2; \overline{g}_1,\overline{g}_2] = d_{C_1, D_1} - d_{C_2, D_2} + d_{C_3, D_3} - d_{C_4, D_4}\in P_{u,v}(B_{u,v}(G))
\end{equation*}

We may therefore assume that $\{u_1,u_2\}\not=\{\overline{v}_1,\overline{v}_2\}$ or $\{v_1,v_2\}\not=\{\overline{u}_1,\overline{u}_2\}$. By symmetry, we may assume that $\overline{u}_1 \not\in \{v_1,v_2\}$. Suppose that $\overline{f}_1\not\in\{f_1,f_2\}$. If $\overline{u}_1$ is a vertex on $C(v_2, v_1)$, let $P$ be a path in $C(v_2, v_1)$ with ends $\overline{u}_1$ and $u_1$. Let $C_1=u\overline{u}_1Pu_1u$
and let $D_1 = vv_1C[v_1, v_2]v_2v$. Orient $C_1$ and $D_1$ such that $P_{u,v}(d_{C_1, D_1}) =  [f_1, \overline{f}_1; g_1, g_2]$.
Then $[f_1,f_2; g_1,g_2]-[\overline{f}_1,f_2; g_1,g_2] = P_{u,v}(d_{C_1, D_1})$, and therefore 
\begin{equation*}
[\overline{f}_1,f_2; g_1,g_2] - [\overline{f}_1, \overline{f}_2; \overline{g}_1,\overline{g}_2]\not\in P_{u,v}(B_{u,v}(G)).
\end{equation*}
However,  $|\{\overline{f}_1,f_2,g_1,g_2\}\cap \{\overline{f}_1,\overline{f}_2,\overline{g}_1,\overline{g}_2\}|$ is larger, a contradiction. The case where $\overline{u}_1$ is a vertex on $C(v_1, v_2)$ is similar. Hence  $\overline{f}_1\in\{f_1,f_2\}$. We assume that $\overline{f}_1=f_1$.

Suppose now that $g_1 \not= \overline{g}_1$ and $g_2\not=\overline{g}_2$. If $\overline{v}_1$ is a vertex on $C(u_1, u_2)$, let $Q$ be a path in $C(u_1, u_2)$ with ends $v_1$ and $\overline{v}_1$. Let $C_1 = uu_2C[u_2,u_1]u_1u$ and let $D_1=vv_1Q\overline{v}_1v$. Orient $C_1$ and $D_1$ such that $P_{u,v}(d_{C_1, D_1}) =  [f_1, f_2; g_1, \overline{g}_1]$. Then $[f_1,f_2; g_1,g_2]-[f_1,f_2; \overline{g}_1,g_2] = P_{u,v}(d_{C_1, D_1})$, and therefore 
\begin{equation*}
[f_1,f_2; \overline{g}_1,g_2] - [f_1, \overline{f}_2; \overline{g}_1,\overline{g}_2]\not\in P_{u,v}(B_{u,v}(G)).
\end{equation*}
However,  $|\{f_1,f_2,\overline{g}_1,g_2\}\cap \{f_1,\overline{u}_2,\overline{g}_1,\overline{g}_2\}|$ is larger, a contradiction. If $\overline{v}_1$ is not a vertex on $C(u_2, u_1)$, then $\overline{v}_2$ is a vertex on $C(u_2, u_1)$. This case is similar. Hence $g_1 = \overline{g}_1$ or $g_2=\overline{g}_2$. We assume that $g_1 = \overline{g}_1$; the case where $g_2=\overline{g}_2$ is similar.

Suppose now that $g_2 \not= \overline{g}_2$ and $f_2 \not= \overline{f}_2$. If $\overline{v}_2$ is a vertex on $C(u_2, u_1)$, let $Q$ be a path in $C(u_2, u_1)$ with ends $v_2$ and $\overline{v}_2$. Let $C_1 = uu_1C[u_1,u_2]u_2u$ and let $D_1 = vv_2Q\overline{v}_2v$. Orient $C_1$ and $D_1$ such that $P_{u,v}(d_{C_1, D_1}) =  [f_1, f_2; g_2, \overline{g}_2]$. Then $[f_1,f_2; g_1,g_2]-[f_1,f_2; g_1,\overline{g}_2] = P_{u,v}(d_{C_1, D_1})$, and therefore 
\begin{equation*}
[f_1,f_2; g_1,\overline{g}_2] - [f_1, \overline{f}_2; g_1,\overline{g}_2]\not\in P_{u,v}(B_{u,v}(G)). 
\end{equation*}
However,  $|\{f_1,f_2,g_1,g_2\}\cap \{f_1,\overline{f}_2,g_1,g_2\}|$ is larger, a contradiction. Hence, $g_2 = \overline{g}_2$ if $\overline{v}_2$ is a vertex on $C(u_2, u_1)$.

If $\overline{v}_2$ is not a vertex on $C(u_2, u_1)$, then $\overline{u}_2$ is a vertex on $C(v_1, v_2)$. Let $P$ be a path in $C(v_1, v_2)$ with ends $u_2$ and $\overline{u}_2$. Let $C_1 = uu_2P\overline{u}_2u$ and let $D_1 = C[v_2, v_1]$. 
Orient $C_1$ and $D_1$ such that $P_{u,v}(d_{C_1, D_1}) =  [f_2, \overline{f}_2; g_1, g_2]$. Then $[f_1,f_2; g_1,g_2]-[f_1,\overline{f}_2; g_1,g_2] = P_{u,v}(d_{C_1, D_1})$, and therefore 
\begin{equation*}
[f_1,f_2; g_1,g_2] - [f_1, f_2; g_1,\overline{g}_2]\not\in P_{u,v}(B_{u,v}(G)). 
\end{equation*}
However,  $|\{f_1,f_2,g_1,g_2\}\cap \{f_1,f_2,g_1,\overline{g}_2\}|$ is larger, a contradiction.

A similar argument now shows that $f_2 = \overline{f}_2$ if in the previous case $\overline{v}_2$ is a vertex on $C(u_2, u_1)$ and that $g_2 = \overline{g}_2$ if in the previous case $\overline{v}_2$ is not a vertex on $C(u_2, u_1)$. 
\end{proof}

.
\begin{lemma}\label{lem:oncycle}
Let $G$ be a graph and let $e=uv$ be an edge of $G$. Suppose $C$ is an oriented  cycle of $G-\{u,v\}$ containing all vertices of $N(\{u,v\})$.
Let $h_1, h_2\in \delta(u)\setminus \{e\}$ and $k_1,k_2\in \delta(v)\setminus \{e\}$ be distinct edges such that the ends $w_1,z_1,w_2,z_2$ of $h_1,k_1,h_2,k_2$ on $C$, respectively, occur in the given order.
Then, for any $2$-cycle $d$ on $G$, there exist a $d'\in B_{u,v}(G)$ and an integer $\alpha$ such that $P_{u,v}(d+d')=\alpha \mat{h_1}{h_2}{k_1}{k_2}.$
\end{lemma}

\begin{proof}
As $N(\{u,v\})$ has at least four vertices, $P_{u,v}(d) = \sum \alpha_i F_i,$ where each $F_i$ is a $4$-cross at $u,v$ and each $\alpha_i$ is an integer.  We show that for each $F_i$, there exist a $c\in B_{u,v}(G)$ and an integer $\alpha_i$ such that $F_i = P_{u,v}(c) + \alpha_i \mat{h_1}{h_2}{k_1}{k_2}$. 

Consider a $4$-cross $F_i = \mat{f_1}{f_2}{g_1}{g_2}$. Let $u_1,u_2$ be the ends of $f_1,f_2$, respectively, that is unequal to $u$, and let $v_1,v_2$ be the ends of $g_1,g_2$, respectively, that is unequal to $v$. If $\{u_1, u_2\}$ and $\{v_1, v_2\}$ do not cross in $C$, then, using the cycle $C$, we can find disjoint cycles $C_1$ and $D_1$ with $u,u_1,u_2\in V(C_1)$ and $v,v_1,v_2\in V(D_1)$. Then $P_{u,v}(d_{C_1, D_1}) = \mat{f_1}{f_2}{g_1}{g_2}$. If $\{u_1, u_2\}$ and $\{v_1, v_2\}$ cross in $C$, then, by Lemma~\ref{lem:uniqueness}, there exists a $c\in B_{u,v}(G)$ such that either 
$\mat{f_1}{f_2}{g_1}{g_2} = \mat{h_1}{h_2}{k_1}{k_2}  + P_{u,v}(c)$ or
 $\mat{f_1}{f_2}{g_1}{g_2} = -\mat{h_1}{h_2}{k_1}{k_2} + P_{u,v}(c).$
\end{proof}

Let $e=uv$ be an edge of a $3$-connected graph such that $G/e$ is $3$-connected. Suppose that $C$ is a cycle of $G-\{u,v\}$ containing all vertices of $N(\{u,v\})$. Let $H_1 = G[\delta(u)\cup \delta(v)\cup E(C)]$, and let $K$ be a subdivision of $K_{3,3}$ in $H_1$ in which $e$ is an arc.
If $d$ is a $2$-cycle on $G$, then, by Lemma~\ref{lem:oncycle}, there exist a $d_1\in B_{u,v}(H_1)$, a $d_2\in B_{v,u}(H_1)$, and integers $\alpha$ and $\beta$ such that 
\begin{align*}
P_{u,v}(d-d_1-d_2) &= \alpha P_{u,v}(d_K),\\ 
P_{v,u}(d-d_1-d_2) &= \beta P_{v,u}(d_K). 
\end{align*}
If we prove that $\alpha=\beta$, then $d' := d-d_1-d_2-\alpha d_K$ satisfies $d'(f,g) = d'(g,f) = 0$ for all $f\in \delta(u)$ and all $g\in \delta(v)$.
In the next sections, we will use topological arguments along with additional conditions on the graph $G$ to show that $\alpha=\beta$.

\section{The Intersection Number}

Let $\phi$ be a drawing of a graph $G=(V,E)$ in the plane that is in  generic position. If $g, f$ are nonadjacent edges $G$, we denote by $\kr_{\phi}(g,f)$ the crossing number of $g$ and $f$ in the drawing $\phi$.  Observe that $\kr_{\phi}(g,f) = -\kr_{\phi}(f,g)$ and $\kr_{\phi}(f,g)\in \mathbb{Z}$ for all pairs $f, g$ of nonadjacent edges of $G$. If $d\in L(G)$, we define 
\begin{equation*}
\kr_{\phi}(d) = \sum_{f,g\in E} \kr_{\phi}(f,g) d(f, g) .
\end{equation*}

\begin{lemma}
Let $\phi_1$ and $\phi_2$ be drawings of a graph $G$ in the plane that are in generic position, and let $d$ be a $2$-cycle on $G$. Then $\kr_{\phi_1}(d) = \kr_{\phi_2}(d)$.
\end{lemma}
\begin{proof}
When going from $\phi_1$ to $\phi_2$, the only changes to $\kr_{\phi_1}(d)$ occur when we pull a vertex $v$ through an edge $e$ that is not incident with $v$, or when we pull an edge $e$ that is not incident with a vertex $v$ through $v$. Since $\sum_{f\in E(G)} [v, f]d(e, f) = 0$ and $\sum_{f\in E(G)} [v, f] d(f, e)=0$ for all $v\in V(G)$, we see that $\kr_{\phi_1}(d) = \kr_{\phi_2}(d)$.
\end{proof}

\begin{lemma}\label{lem:crzero}
Let $\phi$ be a drawing of a graph $G$ in the plane that is in generic position. Then, for any $2$-cycle $d$ on $G$, $\kr_{\phi}(d) = 0$.
\end{lemma}
\begin{proof}
Let $\phi_1$ be a drawing of $G$ in the plane and let $\phi_2$ be the mirror drawing of $\phi_1$. Let $d$ be a $2$-cycle on $G$. Then $\kr_{\phi_1}(d) = -\kr_{\phi_2}(d)$. By the previous lemma, $2\kr_{\phi_1}(d) = 0$. Hence $\kr_{\phi_1}(d)=0$.
\end{proof}

\section{No disjoint paths}

Let $e=uv$ be an edge of a $3$-connected graph such that $G/e$ is $3$-connected.
Suppose that $N(\{u,v\})$ has at least four vertices and that $H := G-\{u,v\}$ has an oriented  cycle $C$ containing all vertices in $N(\{u,v\})$. If $P, Q$ are disjoint paths with ends $r_1,r_2\in V(C)$ and $s_1,s_2\in V(C)$, respectively, and $r_1,s_1,r_2,s_2$ or $r_1,s_2,r_2,s_1$ occur in this order on $C$, we say that $P,Q$ are \emph{$C$-crossing} paths on $C$. Let $d$ be a $2$-cycle on $G$. In this section, we show that if there exist no $d_1\in B_{u,v}(G)$, $d_2\in B_{v,u}(G)$, and quad $2$-cycles $q_i$ on $G$, $i=1,\ldots,k$, such that $d' = d-d_1-d_2-\sum_{i=1}^k q_i$ satisfies $P_{u,v}(d') = 0$ and $P_{v,u}(d') = 0$, then there are no two disjoint $C$-crossing paths in $H$ with ends in $N(\{u,v\})$.

Let $C$ be an oriented cycle of a graph $G$.
If $u, v\in V(C)$ are distinct, we say that $A$ and $B$ \emph{cross in $C(u,v)$} if there exist distinct vertices $v_1,v_2,v_3,v_4\in V(C(u,v))$, with $v_1,v_3\in A$ and $v_2,v_4\in B$ such that $v_1,v_2,v_3,v_4$ occur in this order in $C(u,v)$. We make similar definitions for $C[u,v)$, $C(u,v]$, and $C[u,v]$.

\begin{lemma}\label{lem:tonocross}
Let $G$ be a graph, let $e=uv$ be an edge of $G$, and let $H=G-\{u,v\}$. Suppose $C$ is an oriented  cycle of $H$ containing all vertices of $N(\{u,v\})$. Let $f_1,f_2\in \delta(u)\setminus\{e\}$ and $g_1, g_2\in \delta(v)\setminus \{e\}$ be distinct edges, where the ends $u_1,v_1,u_2,v_2$ of $f_1,g_1,f_2,g_2$ on $C$, respectively, occur in the given order.
Then one of the following holds:
\begin{enumerate}
\item there exists a $d\in B_{u,v}(G)$ with $P_{u,v}(d) = \mat{f_1}{f_2}{g_1}{g_2}$;
\item there exist a quad $2$-cycle $q$ and a $d\in B_{u,v}(G)$ with $P_{u,v}(q+d) = \mat{f_1}{f_2}{g_1}{g_2}$; or
\item there are no two disjoint $C$-crossing paths in $H$ with ends in $N(\{u,v\})$.
\end{enumerate}
\end{lemma}
\begin{proof}
Suppose that there are two disjoint $C$-crossing paths in $H$ with ends in $N(\{u,v\})$; that is, there are disjoint paths $P, Q$ in $H$ with ends $r_1,r_2\in N(\{u,v\})$ and $s_1,s_2\in N(\{u,v\})$, respectively, such that $r_1,s_1,r_2,s_2$ occur in $C$ in order.
We may assume that $P$ and $Q$ are internally disjoint from $N(\{u,v\})$, and we assume that $s_1$ is on $C(r_1, r_2)$.

If $r_1, r_2\in N(u)$ and $s_1, s_2\in N(v)$, let $C_1:=ur_1Pr_2u$ and let $D_1:=vs_1Qs_2v$. Orient $C_1$ and $D_1$ such that $P_{u,v}(d_{C_1, D_1}) = \mat{ur_1}{ur_2}{vs_1}{vs_2}.$ Since the vertices $r_1,s_1,r_2,s_2$ occur in the same order as $u_1,v_1,u_2,v_2$ on $C$, by Lemma~\ref{lem:uniqueness}, there exists a $d\in B_{u,v}(G)$ such that $P_{u,v}(d_{C_1, D_1}+d) = \mat{f_1}{f_2}{g_1}{g_2};$ that is, $\mat{f_1}{f_2}{g_1}{g_2}\in P_{u,v}(B_{u,v}(G)).$ In the same way, if  $r_1, r_2\in N(v)$ and $s_1, s_2\in N(u)$, then $\mat{f_1}{f_2}{g_1}{g_2} \in P_{u,v}(B_{u,v}(G)).$ 

Next suppose that $r_1,r_2,s_1, s_2\in N(\{u,v\})\setminus N(v)$. Suppose that at least two of the segments $C(r_1,s_1)$, $C(s_1, r_2)$, $C(r_2, s_2)$, $C(s_2, r_1)$ contain vertices of $N(v)$. We consider the case where $C(r_1,s_1)$ and $C(s_2,r_1)$ contain vertices of $N(v)$; the other cases are similar. Let $v'\in N(v)$ be the nearest vertex to $s_1$ in $C(r_1,s_1]$, and let $v''\in N(v)$ be the nearest vertex  to $s_2$ in $C[s_2,r_1)$. Let 
$C_1 :=ur_1Pr_2u$ and let $D_1 := vv'C[v', s_1]QC[s_2,v'']v''v.$
Orient $C_1$ and $D_1$ such that $d_{C_1,D_1}(ur_1,vv') = 1$. Then 
$$P_{u,v}(d_{C_1, D_1}) = \mat{ur_1}{ur_2}{vv'}{vv''}.$$ 
Since the vertices $r_1,v',r_2,v''$ occur in the same order as $u_1,v_1,u_2,v_2$ on $C$, by Lemma~\ref{lem:uniqueness}, there exists a $d\in B_{u,v}(G)$ such that $P_{u,v}(d) = \mat{f_1}{f_2}{g_1}{g_2}.$

So we may assume that at most one segment contains vertices of $N(v)$. By symmetry, we may assume that $C(r_1, s_1)$ contains vertices of $N(v)$. Let $v'\in N(v)$ be the nearest vertex to $r_1$ in $C[r_1,s_1)$ and let $v''\in N(v)$ be the nearest vertex to $s_1$ in $C(r_1,s_1]$. As $N_H(u)$ and $N_H(v)$ cross in $C$, there exists a vertex $u'\in N_H(u)$ between $v'$ and $v''$ on $C(r_1,s_1)$. Let $K$ be the quad in $$C[s_2, r_2]\cup P\cup Q\cup G[\{v'v, v''v, r_2u, s_2u, u'u, uv\}].$$ 
Let $A$ be the arc in $K$ containing the edge $ur_2$, and let $\ell$ be the end of $A$ unequal to $u$. 
Let $q$ be a quad on $K$ with left side $\{\ell,u'\}$ and $q(r_2u, v'v) = 1$. Then 
$$P_{u,v}(q) = \mat{ur_2}{uu'}{vv'}{vv''}.$$ 
Since the vertices $r_2,v',u',v''$ occur in the same order as $u_1,v_1,u_2,v_2$ on $C$, by Lemma~\ref{lem:uniqueness}, there exists a $d\in B_{u,v}(G)$ such that $P_{u,v}(q + d) = \mat{f_1}{f_2}{g_1}{g_2}.$ The case where $r_1,r_2,s_1, s_2\in N(\{u,v\})\setminus N(u)$ is similar.

Therefore, we may assume that at least one of the paths $P, Q$ has one end in $N(u)$ and the other end in $N(v)$. 
By symmetry, we may assume that $P$ has one end in $N(u)$ and the other end in $N(v)$. We first assume that $Q$ have both ends in $N(u)$ or in $N(v)$.
By symmetry, we may assume that $r_1\in N(v)$ and $r_2\in N(u)$, and that both ends of $Q$ belong to $N(v)$. We may assume that $r_1\not\in N(u)$, for otherwise we are in the first case above.
Suppose $C(s_2,s_1)$ contains a vertex $u'\in N(u)$. If $u'$ is on $C(s_2,r_1)$, let 
$C_1 := uu'C[u', r_1]Pr_2u$ and let $D_1 := vs_1Qs_2v.$
Orient $C_1$ and $D_1$ such that $d_{C_1,D_1}(uu',vs_1) = 1$.
Then $$P_{u,v}(d_{C_1, D_1}) := \mat{uu'}{ur_2}{vs_1}{vs_2} \in P_{u,v}(B_{u,v}(G)).$$
Since the vertices $u',s_1,r_2,s_2$ occur in the same order as $u_1,v_1,u_2,v_2$ on $C$, by Lemma~\ref{lem:uniqueness}, there exists a $d\in B_{u,v}(G)$ such that $P_{u,v}(d) = \mat{f_1}{f_2}{g_1}{g_2}.$
In the same way, if $u'$ is on $C(r_1, s_1)$, then $\mat{f_1}{f_2}{g_1}{g_2}  \in P_{u,v}(B_{u,v}(G)).$ Hence, we may assume that $C(s_2,s_1)$ contains no vertex of $N(u)$. Suppose both $C[s_1, r_2)$ and $C(r_2,s_2]$ contain vertices of $N(u)$. Let $u'\in N(u)$ be the nearest vertex  to $s_1$ on $C[s_1, r_2)$ and let $u''\in N(u)$ be the nearest vertex to $s_2$ on $C(r_2, s_2]$. As $N_H(u)$ and $N_H(v)$ cross in $C$, there exists a vertex $v'\in N_H(v)$ between $u'$ and $u''$. Let $T$ be the path in $C(u',u'')$ from $r_2$ to $v'$.
Let  
$C_1 := uu''C[u'', s_2]QC[s_1, u']u'u$ and let $D_1 := vr_1PTv'v.$ 
Orient $C_1$ and $D_1$ such that $d_{C_1,D_1}(uu'',vr_1) = 1$.
Then 
$$P_{u,v}(d_{C_1, D_1}) = \mat{uu''}{uu'}{vr_1}{vv'} \in B_{u,v}(G).$$ 
Since $u'', r_1,u', v'$ occur in the same order as $u_1,v_1,u_2,v_2$ on $C$, by Lemma~\ref{lem:uniqueness}, there exists a $d\in B_{u,v}(G)$ such that $P_{u,v}(d) = \mat{f_1}{f_2}{g_1}{g_2}.$
Hence at least one of $C[s_1, r_2)$ and $C(r_2,s_2]$ has only vertices from $N(v)$. We may assume that $C(r_2,s_2]$ has only vertices from $N(v)$. Since $N_H(u)$ and $N_H(v)$ cross in $C$, there exist a vertex of $N(v)$ on $C(s_1,r_2)$ and a vertex of $N(u)$ on $C[s_1,r_2)$. Let $v'\in N(v)\setminus \{r_2\}$ nearest to $r_2$ on $C(s_1,r_2]$, and let $u'\in N(u)$ be the nearest vertex to $s_1$ on $C[s_1,v')$. Let $K$ be the quad in $$(C-V(C(s_2, r_1)))\cup P\cup Q\cup G[\{vr_1, vs_2, vv', uv, ur_2, uu'\}].$$ Let $A$ be the arc in $K$ containing the edge $vr_1$, and let $\ell$ be the end of $A$ unequal to $v$. Let $q$ be a quad $2$-cycle on $K$ with left side $\{\ell,u'\}$ and $q(uu', vr_1) = 1$. Then $$P_{u,v}(q) = \mat{ur_2}{uu'}{vv'}{vr_1}.$$ Since $r_2, v', u', r_1$ occur in the same order as $u_1,v_1,u_2,v_2$ on $C$, by Lemma~\ref{lem:uniqueness}, there exists a $d\in B_{u,v}(G)$ such that $P_{u,v}(q+d) = \mat{f_1}{f_2}{g_1}{g_2}.$ 

We may therefore assume that both $P$ and $Q$ have one end in $N(u)$ and the other end in $N(v)$, but both ends do not belong to $N(u)$ or to $N(v)$. We assume that $r_1, s_1\in N(u)$, and $r_2, s_2\in N(v)$, the other cases are similar. Suppose $N_H(u)$ and $N_H(v)$ cross in $C(s_2, r_1)$. Let $v'\in N(v)$ be the nearest vertex to $r_1$ on $C(s_2, r_1)$, and let $u'\in N(u)$ be the nearest vertex to $s_2$ on $C(s_2, r_1)$. Let 
$C_1 := us_1QC[s_2,u']u'u$ and let
$D_1 := vv'C[v',r_1]Pr_2v.$
Orient $C_1$ and $D_1$ so that $d_{C_1, D_1}(uu', vv') = 1$. Then 
$$P_{u,v}(d_{C_1, D_1}) = \mat{uu'}{us_1}{vv'}{vr_2} \in B_{u,v}(G).$$ 
Since the vertices $u',v',s_1,r_2$ occur in the same order as $u_1,v_1,u_2,v_2$ on $C$, by Lemma~\ref{lem:uniqueness}, there exists a $d\in B_{u,v}(G)$ such that
$P_{u,v}(d_{C_1,D_1} + d) = \mat{f_1}{f_2}{g_1}{g_2} \in P_{u,v}(B_{u,v}(G)).$ Hence, we may assume that $N_H(u)$ and $N_H(v)$ do not cross in $C(s_2, r_1)$. In the same way, we may assume that $N_H(u)$ and $N_H(v)$ do not cross in $C(s_1, r_2)$. If $C[r_1, s_1]$ contains a vertex $v'$ of $N(v)$ and $C[r_2, s_2]$ contains a vertex $u'$ of $N(u)$, then $\mat{ur_1}{uu'}{vv'}{vs_2}\in P_{u,v}(B_{u,v}(G))$.  Since the vertices $r_1,v',u',s_2$ occur in the same order as $u_1,v_1,u_2,v_2$ on $C$, by Lemma~\ref{lem:uniqueness}, $\mat{f_1}{f_2}{g_1}{g_2} \in P_{u,v}(B_{u,v}(G)).$

We may therefore assume that  $C[r_1, s_1]$ has no vertices of $N(v)$ or that $C[r_2, s_2]$ has no vertices of $N(u)$; by symmetry, we may assume that $C[r_1, s_1]$ has no vertices of $N(v)$. Since $N_H(u)$ and $N_H(v)$ cross in $C$, $N_H(u)$ and $N_H(v)$ cross in $C[r_2, s_2]$. Let $u'$ be the vertex on $C(r_2, s_2)$ nearest to $r_2$, and let $v'$ be the vertex on $C(u', s_2)$ nearest to $u'$. Let $K$ be the quad in $$(C-V(C(r_1, s_1)))\cup P\cup Q \cup G[\{ur_1, us_1, uu', uv, vr_2, vv'\}].$$ Let $A$ be the arc in $K$ containing the edge $ur_1$, and let $\ell$ be the end of $A$ unequal to $u$. Let $q$ be a quad $2$-cycle on $K$ with left side $\{\ell, u'\}$ and $q(ur_1, vr_2) = 1$. Then $$P_{u,v}(q) = \mat{ur_1}{uu'}{vr_2}{vv'}.$$ Since $r_1, r_2, u', v'$ occur in the same order as $u_1,v_1,u_2,v_2$ on $C$, by Lemma~\ref{lem:uniqueness}, there exists a $d\in B_{u,v}(G)$ such that $P_{u,v}(q+d) = \mat{f_1}{f_2}{g_1}{g_2}.$
\end{proof}

From Lemmas~\ref{lem:tocycle} and~\ref{lem:tonocross}, we obtain the following lemma.
\begin{lemma}\label{lem:cyclenodisjoint}
Let $e = uv$ be an edge of a $3$-connected graph $G$ such that $G/e$ is $3$-connected. Suppose $N(\{u,v\})$ has at least four vertices. Then one of the following holds:
\begin{enumerate}
\item for any $2$-cycle $d$, there exist a $d_1\in B_{u,v}(G)$, a $d_2\in B_{v,u}(G)$, and quad $2$-cycles $q_i$ on $G$, $i=1,\ldots,k$, such that $z = d-d_1-d_2-\sum_{i=1}^k q_i$ satisfies $P_{u,v}(z) = 0$ and $P_{v,u}(z) = 0$; or
\item there exists an oriented cycle $C$ in $H:=G-\{u,v\}$ containing all vertices of $N(\{u,v\})$, and  $N_H(u)$ and $N_H(v)$ cross in $C$, and there are no two disjoint $C$-crossing paths in $H$ with ends in $N(\{u,v\})$.
\end{enumerate}
\end{lemma}

\section{No tripod}

Let $v_1,v_2,v_3$ be distinct vertices of a graph $G$. A \emph{tripod} on $v_1,v_2,v_3$ is a subgraph $P_1\cup R_1\cup P_2\cup R_2\cup P_3\cup R_3\cup Q_1\cup Q_2\cup Q_3$ of $G$ 
consisting of
\begin{enumerate}
\item two vertices $a,b$ so that $a,b,v_1,v_2,v_3$ are all distinct;
\item three paths $P_1\cup R_1$, $P_2\cup R_2$, $P_3\cup R_3$ of $G$ between $a$ and $b$, mutually internally disjoint, each at least with one internal vertex;
\item  three paths $Q_1,Q_2,Q_3$ of $G$, mutually disjoint, such that for $i=1,2,3$, $Q_i$ has ends $u_i$ and $v_i$, where $u_i\in V(P_i\cup R_i)\setminus \{a,b\}$, and no vertex of $Q_i$ except for $u_i$ belongs to $V(P_1\cup R_1\cup P_2\cup R_2\cup P_3\cup R_3)$.
\end{enumerate}
The paths $Q_1,Q_2,Q_3$ are called the \emph{legs} of the tripod, the vertices $v_1,v_2,v_3$ the \emph{feet}, and the vertices $a, b$ the \emph{midpoints}.

Let $e=uv$ be an edge of a $3$-connected graph $G=(V,E)$ such that $G/e$ is $3$-connected. Let $H := G-\{u,v\}$. Suppose $C$ is a cycle of $H$ containing all vertices of $N(\{u,v\})$. If there are no two disjoint $C$-crossing paths in $H$ with ends in $N(\{u,v\})$, and there is no tripod in $H$ with feet in $N(\{u,v\})$, we use the following theorem; see \cite{RobertsonST93}.

\begin{theorem}\label{thm:disjointpaths}
Let $v_1,\ldots,v_k$ be distinct vertices of a graph $G$. If there is no $(\leq 2)$-separation $(G_1,G_2)$ with $v_1,\ldots,v_k\in V(G_1)$ and $V(G_1)\not=V(G)$, then at least of the following holds:
\begin{enumerate}
\item there are disjoint paths of $G$ with ends $p_1,p_2$ and $q_1,q_2$, respectively, such that $p_1,q_1,p_2,q_2$ occur in the sequence $v_1,\ldots,v_k$ in order,
\item there exists a tripod with feet in $\{v_1,\ldots,v_k\}$, or
\item $G$ can be drawn in a disc with $v_1,\ldots,v_k$ on the boundary in order.
\end{enumerate}
\end{theorem}

\begin{lemma}\label{lem:4notripod}
Let $e=uv$ be an edge of a $3$-connected graph $G=(V,E)$ such that $G/e$ is $3$-connected. Suppose $H := G-\{u,v\}$ has an oriented cycle $C$ containing all vertices of $N(\{u,v\})$ such that $N_H(u)$ and $N_H(v)$ cross in $C$.
Suppose there are no two disjoint $C$-crossing paths in $H$ with ends in $N(\{u,v\})$,
%
and there is no tripod in $H$ with feet in $N(\{u,v\})$. Then,
for any $2$-cycle $d$ on $G$, there exist a $d_1\in B_{u,v}(G)$, a $d_2\in B_{v,u}(G)$, a Kuratowski $2$-cycle $d_K$ on a $K_{3,3}$-subdivision $K$ in $G$, and an integer $\alpha$ such that $z = d-d_1-d_2- \alpha d_K$ satisfies $P_{u,v}(z) = 0$ and $P_{v,u}(z) = 0$.
\end{lemma}
\begin{proof}
Let $d$ be a $2$-cycle on $G$.
Since $N_H(u)$ and $N_H(v)$ cross in $C$, there exists a $K_{3,3}$-subdivision $K$
in $G[\delta(u)\cup \delta(v)\cup E(C)]$ in which $e$ is an arc.

 As $G/e$ is $3$-connected, there is no $(\leq 2)$-separation $(H_1, H_2)$ of $H$ with $N(\{u,v\})\subseteq V(H_1)$ and $V(H_1)\not=V(H)$. As there are no disjoint paths in $H$ with ends $p_1,p_2\in N(\{u,v\})$ and $q_1,q_2\in N(\{u,v\})$, respectively, such that $p_1,q_1,p_2,q_2$ occur in $C$ in order, and $H$ has no tripod with feet in $N(\{u,v\})$, $H$ can be embedded in the unit disc $D$ with the vertices of $N(\{u,v\})$ on the boundary in the order induced by $C$, by Theorem~\ref{thm:disjointpaths}. Extend the embedding of $H$ in $D$ to a drawing $\phi$ of $G$ in the plane by mapping the vertices $u$ and $v$, and the interiors of the edges incident with $u$ and $v$ in generic position in the complement of $D$. This can be done such that $K$ has only two edges $h_1,h_2$ that intersect; we assume that $h_1\in \delta(u)$ and $h_2\in \delta(v)$. 
  
We can choose a Kuratowski $2$-cycle $d_K$ on $K$ such that $d_K(h_1,h_2) = 1$.
By Lemma~\ref{lem:oncycle}, there exist a $d_1\in B_{u,v}(H_1)$ and a $d_2\in B_{v,u}(H_1)$ such that, if $d' =  d-d_1-d_2$, then 
\begin{align*}
P_{u,v}(d') &= \alpha P_{u,v}(d_K) \text{ for some integer $\alpha$,}\\
P_{v,u}(d') &= \beta P_{v,u}(d_K) \text{ for some integer $\beta$.}
\end{align*}
As the restriction of the drawing $\phi$ to $H$ is an embedding of $H$ in $D$, $\kr_{\phi}(d'_{\restriction E(H)^2})=0$.  Let $H' := G[\delta(u)\cup\delta(v)]$. Since $\kr_{\phi}(d) = 0$ for any $2$-cycle $d$ on $G$, $\kr_{\phi}(d'_{\restriction E(H')^2}) = 0$. 
Hence $$\kr_{\phi}(\alpha P_{u,v}(d_K) + \beta P_{v,u}(d_K)) = 0,$$ that is, $$\alpha \kr_{\phi}(P_{u,v}(d_K))+ \beta  \kr_{\phi}(P_{v,u}(d_K)) = 0.$$ Since $\kr_{\phi}(P_{u,v}(d_K)) = \kr_{\phi}(h_1,h_2)$, $\kr_{\phi}(P_{v,u}(d_K)) = \kr_{\phi}(h_2,h_1)$, and $\kr_{\phi}(h_1,h_2) = -\kr_{\phi}(h_2,h_1)$, we obtain that $\alpha=\beta$. Hence $z = d-d_1-d_2-\alpha d_K$ satisfies $P_{u,v}(z) = 0$ and $P_{v,u}(z) = 0$.
\end{proof}

\begin{lemma}\label{lem:3notripod}
Let $e=uv$ be an edge of a $3$-connected graph $G=(V,E)$ such that $G/e$ is $3$-connected. Suppose $N(\{u,v\})$ has exactly three vertices and $H:=G-\{u,v\}$ has no tripod with feet in $N(\{u,v\})$. Then, for any $2$-cycle $d$ on $G$, there exist a Kuratowski $2$-cycle $d_K$ on a Kuratowski subgraph $K$ in $G$ and an integer $\alpha$ such that $z = d- \alpha d_K$ satisfies $P_{u,v}(z) = 0$ and $P_{v,u}(z) = 0$..
\end{lemma}
\begin{proof}
Let $d$ be a $2$-cycle on $G$. As $N(\{u,v\})$ has exactly three vertices, $P_{u,v}(d) = \alpha[f_1,f_2,f_3; g_1, g_2, g_3]$ for some integer $\alpha$, where $f_1,f_2,f_3\in \delta(u)\setminus\{e\}$ and $g_1,g_2,g_3\in \delta(v)\setminus\{e\}$.  Let $u_1,u_2,u_3$ be the three vertices in $N(\{u,v\})$. 

As $G/e$ is $3$-connected, there is no $(\leq 2)$-separation $(H_1, H_2)$ of $H$ with $N(\{u,v\})\subseteq V(H_1)$ and $V(H_1)\not=V(H)$. By Theorem~\ref{thm:disjointpaths}, $H$ can be embedded in a disc $D$ with the vertices of $N(\{u,v\})$ on the boundary. Then $\kr_{\phi}(d_{\restriction E(H)^2}) = 0$. Let $H' := G[\delta(u)\cup\delta(v)]$. Extend the embedding of $H$ in $D$ to a drawing $\phi$ of $G$ in the plane by mapping the vertices $u$ and $v$, and the interiors of the edges incident with $u$ and $v$ in generic position in the complement of $D$. This can be done such that $H'$ has only two nonadjacent edges $h_1,h_2$ that intersect; we assume that $h_1\in \delta(u)$ and $h_2\in \delta(v)$.
Then $\kr_{\phi}(d_{\restriction E(H')^2}) = 0$, by Lemma~\ref{lem:crzero}. 

Suppose $H-N(\{u,v\})$ contains a vertex $z$. 
As there is no $(\leq 2)$-separation $(H_1, H_2)$ of $H$ with $N(\{u,v\})\subseteq V(H_1)$ and $V(H_1)\not=V(H)$, there exist three internally disjoint paths $P_1, P_2, P_3$ starting at $u_1,u_2,u_3$, respectively, and ending at $z$. Then $$K := G[\delta(u)\setminus\{e\})\cup (\delta(v)\setminus\{e\})\cup E(P_1)\cup E(P_2)\cup E(P_3)]$$ is a subdivision of $K_{3,3}$.
Since $h_1$ and $h_2$ belong
to disjoint subdivided edges of $K_{3,3}$, we can choose the
Kuratowski $2$-cycle $d_K$ on $K$ such that $d_K(h_1,h_2) = 1$. Then 
\begin{align*}
P_{u,v}(d) &= \alpha P_{u,v}(d_K) \text{ for some integer $\alpha$,}\\
P_{v,u}(d) &= \beta P_{v,u}(d_K) \text{ for some integer $\beta$.}
\end{align*}
Since $\kr_{\phi}(d_{\restriction E(H')^2}) = \alpha-\beta$ and $\kr_{\phi}(d_{\restriction E(H')^2}) = 0$, $\alpha=\beta$.  So $z = d-\alpha d_K$ satisfies $P_{u,v}(z) = 0$ and $P_{v,u}(z) = 0$.

If $H-\{u_1,u_2,u_3\}$ contains no vertex, then, by the $2$-connectivity of $H$, $H$ is a triangle on $\{u_1,u_2,u_3\}$. Then $$K := G[\delta(u)\cup \delta(v)\cup \{u_1u_2, u_2u_3, u_3u_1\}]$$ is a $K_5$. Since $h_1$ and $h_2$ belong
to nonadjacent edges of $K_{5}$, we can choose a
Kuratowski $2$-cycle $d_K$ on $K$ such that $d_K(h_1,h_2) = 1$. 
Then 
\begin{align*}
P_{u,v}(d) &= \alpha P_{u,v}(d_K) \text{ for some integer $\alpha$,}\\
P_{v,u}(d) &= \beta P_{v,u}(d_K) \text{ for some integer $\beta$.}
\end{align*}
Since $\kr_{\phi}(d_{\restriction E(H')^2}) = 0$, $\alpha=\beta$.  So $z = d-\alpha d_K$ satisfies $P_{u,v}(z) = 0$ and $P_{v,u}(z) = 0$.
\end{proof}

\section{Using a tripod}

In this section, we study the case where $H$ has a tripod with feet in $N(\{u,v\})$.

\begin{lemma}
Let $v_1,\ldots,v_k$ be distinct vertices of a graph $G$. Suppose $G$ has a tripod with feet in $\{v_1,\ldots,v_k\}$. If there are no two disjoint paths of $G$ with ends $p_1,p_2$ and $q_1,q_2$, respectively, such that $p_1,q_1,p_2,q_2$ occur in the sequence $v_1,\ldots,v_k$ in order, then there is a $3$-separation $(G_1, G_2)$ of $G$ with $\{v_1,\ldots,v_k\}\subseteq V(G_2)$ such that $G_1$ has a tripod with feet in $V(G_1\cap G_2)$, and there exist three disjoint paths from $V(G_1\cap G_2)$ to $\{v_1,\ldots,v_k\}$.
\end{lemma}

\begin{lemma}\label{lem:2tripods}
Let $e = uv$ be an edge of a $3$-connected graph $G$ such that $G/e$ is $3$-connected.  Suppose $H := G-\{u,v\}$ has an oriented cycle $C$ containing all vertices of $N(\{u,v\})$ such that $N_H(u)$ and $N_H(v)$ cross in $C$. 
If there are no two disjoint $C$-crossing paths in $H$ with ends in $N(\{u,v\})$, but there is a tripod in $H$ with feet in $N(\{u,v\})$, then there exist a tripod $TP_1$ in $H$ with feet, $p_1,p_2,p_3$, in $V(C)$ and a tripod $TP_2$ in $H' := G[\delta(u)\cup \delta(v)\cup E(C)]$ with feet $p_1,p_2,p_3$ such that $V(TP_1\cap TP_2)=\{p_1,p_2,p_3\}$.
\end{lemma}
\begin{proof}
Since there are no disjoint paths in $H$ with ends $r_1,r_2\in N(\{u,v\})$ and $s_1,s_2\in N(\{u,v\})$, respectively, such that $r_1,s_1,r_2,s_2$ occur in $C$ in order, but $H$ has a tripod with feet in $N(\{u,v\})$, there exists a $3$-separation $(H_1,H_2)$ of $H$ with $N(\{u,v\})\subseteq V(H_2)$ such that $H_1$ has a tripod $TP'$ with feet $V(H_1\cap H_2)$, and there are three disjoint paths $P_1,P_2,P_3$ from $V(H_1\cap H_2)$ to $N(\{u,v\})$. For $i=1,2,3$, let $p_i$ be the first vertex of $P_i$ on $C$ when going from the end in $V(H_1\cap H_2)$ to $N(\{u,v\})$. Then $TP_1 := TP'\cup P_1\cup P_2\cup P_3$ is a tripod of $H$ with feet $p_1,p_2,p_3$. We may assume that $p_1,p_2,p_3$ appear in this order on $C$.

Suppose first that $V(H_1)\setminus V(H_2)$ contains no vertices of $C$. We may assume that $C\subseteq H_2$. Observe that $C(p_1,p_3)$ contain at least one vertex of $N_H(\{u,v\})$, as $P_2$ ends in $N(\{u,v\})$. In the same way $C(p_2,p_1)$ and $C(p_3, p_2)$ each contain at least one vertex of $N_H(\{u,v\})$.  Let $F'$ be obtained from $H'$ by suppressing any vertex of degree two that is not in $\{p_1,p_2,p_3\}$. If $F'$ has a tripod with feet $p_1,p_2,p_3$, then $H'$ has one.  Any vertex of $V(F')\setminus \{p_1,p_2,p_3,u,v\}$ belongs to $N(\{u,v\})$. Suppose now that there is a $(\leq 2)$-separation $(F_1, F_2)$ of $F'$ with $\{p_1,p_2,p_3\}\subseteq V(F_2)$ and $V(F_2)\not=V(F')$. Let $s\in V(F')\setminus V(F_2)$. If $s\not\in \{u,v\}$, then $s\in N(\{u,v\})$, and there are three internally disjoint paths from $s$ to $p_1,p_2,p_3$, as each of $C(p_1,p_3)$, $C(p_2,p_1)$, and $C(p_3,p_2)$ contains a vertex of $N(\{u,v\})$; a contradiction. The case $s\in \{u,v\}$ is analogous. If $F'$ can be embedded in a disc with $p_1,p_2,p_3$ on the boundary, then the graph $H'$ can be embedded in the plane; a contradiction. Hence $F'$ has a tripod with feet $p_1,p_2,p_3$ by Theorem~\ref{thm:disjointpaths}, and so does $H'$ have a tripod $TP_2$ with feet $p_1,p_2,p_3$. 

Suppose next that $V(H_1)\setminus V(H_2)$ contains vertices of $C$. There are vertices $t_1,t_2\in V(C)$ such that $V(C[t_1,t_2])$ contains at least three vertices, $C[t_1,t_2]\subseteq H_1$, and any $t\in V(C(t_2,t_1))$ does not belong to $V(H_1)$. Then $t_1,t_2\in \{p_1,p_2,p_3\}$. 

Suppose first that $C(t_1,t_2)$ contains no vertices of $\{p_1,p_2,p_3\}$. By symmetry, we may assume that $t_1 = p_1$ and $t_2 = p_2$. 
Consider now the graph $F = G[E(C[p_2,p_1])\cup \delta(u)\cup \delta(v)]$. We prove that $F$ has a tripod with feet $\{p_1,p_2,p_3\}$. Let $F'$ be obtained from $F$ by suppressing any vertex of degree two that is not in $\{p_1,p_2,p_3\}$. If $F'$ has a tripod with feet $p_1,p_2,p_3$, then $F$ has one.  Any vertex of $V(F')\setminus \{p_1,p_2,p_3,u,v\}$ belongs to $N(\{u,v\})$. Observe that each of $C[p_2,p_3)$, $C(p_2,p_1)$, and $C(p_3,p_1]$ contains at least one vertex of $N(\{u,v\})$. Suppose to the contrary that there is a $(\leq 2)$-separation $(F_1,F_2)$ of $F'$ with $p_1,p_2,p_3\in V(F_2)$ and $V(F_2)\not=V(F')$. Let $s\in V(F')\setminus V(F_2)$.  If $s\not\in \{u,v\}$, then $s\in N(\{u,v\})$, and there are three internally disjoint paths from $s$ to $p_1,p_2,p_3$; a contradiction. The case $s\in \{u,v\}$ is analogous. If $F'$ can be embedded in a disc with $p_1,p_2,p_3$ on the boundary, then the graph $G[E(C)\cup \delta(u)\cup \delta(v)]$ can be embedded in the plane; a contradiction. Hence $F'$ has a tripod with feet $p_1,p_2,p_3$ by Theorem~\ref{thm:disjointpaths}, and so does $F$ have a tripod $TP_2$ with feet $p_1,p_2,p_3$. 

Suppose next that $C(t_1,t_2)$ contains a vertices of $\{p_1,p_2,p_3\}$. By symmetry, we may assume that $t_1 = p_1$ and $t_2 = p_3$. Then $p_2\in N(\{u,v\})$.
Consider now the graph $F = G[E(C[p_3,p_1])\cup \delta(u)\cup \delta(v)]$. We prove that $F$ has a tripod with feet $\{p_1,p_2,p_3\}$. Let $F'$ be obtained from $F$ by suppressing any vertex of degree two that is not in $\{p_1,p_2,p_3\}$. If $F'$ has a tripod with feet $p_1,p_2,p_3$, then $F$ has one.  Any vertex of $V(F')\setminus \{p_1,p_2,p_3,u,v\}$ belongs to $N(\{u,v\})$. Suppose that there is a $(\leq 2)$-separation $(F_1,F_2)$ of $F'$ with $p_1,p_2,p_3\in V(F_2)$ and $V(F_2)\not=V(F')$. Let $s\in V(F')\setminus V(F_2)$. If $s\not\in \{u,v\}$, then $s\in N(\{u,v\})$, and there are three internally disjoint paths from $s$ to $p_1,p_2,p_3$; a contradiction. The case where $s\in\{u,v\}$ is analogous. If $F'$ can be embedded in a disc with $p_1,p_2,p_3$ on the boundary, then the graph $G[E(C)\cup \delta(u)\cup \delta(v)]$ can be embedded in the plane; a contradiction. Hence $F'$ has a tripod with feet $p_1,p_2,p_3$ by Theorem~\ref{thm:disjointpaths}, and so does $F$ have a tripod $TP_2$ with feet $p_1,p_2,p_3$. 
\end{proof}

\begin{lemma}
Let $TP_1$ and $TP_2$ be tripods that have only their feet in common, and let $a$ and $b$ be midpoints of $TP_1$ and $TP_2$, respectively. 
Let $TP_1$ be drawn in a disc $D_1$ with the feet of $TP_1$ on the boundary of $D_1$. Let $q$ be a quad $2$-cycle on $TP_1\cup TP_2$ with left side $\{a,b\}$. Then $\kr(q_{\restriction E(TP_1)^2}) \in \{-1, 1\}$.
\end{lemma}

\begin{lemma}\label{lem:4hastripod}
Let $e = uv$ be an edge of a $3$-connected graph $G$ such that $G/e$ is $3$-connected.  Suppose $H := G-\{u,v\}$ has an oriented cycle $C$ containing all vertices of $N(\{u,v\})$, and $N_H(u)$ and $N_H(v)$ cross in $C$. If there are no two disjoint $C$-crossing paths in $H$ with ends in $N(\{u,v\})$, but there is a tripod in $H$ with feet in $N(\{u,v\})$, then, for any $2$-cycle $d$ on $G$, there exist a $d_1\in B_{u,v}(G)$, a $d_2\in B_{v,u}(G)$, a quad $2$-cycle $q$, an integer $\alpha$, a Kuratowski $2$-cycle $d_K$ on a $K_{3,3}$-subdivision $K$ in $G$, and an integer $\beta$, such that $z = d - d_1-d_2 - \alpha q - \beta d_K$ satisfies $P_{u,v}(z) = 0$ and $P_{v,u}(z) = 0$..
\end{lemma}
\begin{proof}
Let $H' = G[\delta(u)\cup \delta(v)\cup E(C)]$. Let $K$ be a subdivision of $K_{3,3}$ in $H'$ in which $e$ is an arc.

By the previous lemma, there exist a tripod $TP_1$ in $H$ with feet, $p_1,p_2,p_3$, in $V(C)$ and a tripod $TP_2$  in $G[\delta(u)\cup \delta(v)\cup E(C)]$ with feet $p_1,p_2,p_3$ such that $V(TP_1\cap TP_2) = \{p_1,p_2,p_3\}$.
Then $TP_1\cup TP_2$ is a quad $Q$ in $G$. Choose midpoints $u_1,u_2$ in $TP_1$ and $TP_2$, respectively. Let $q$ be a quad $2$-cycle on $Q$ with left side $\{u_1,u_2\}$.

Let $\phi$ be a drawing of $G$ in the plane that is in generic position such that $H$ is drawn in a disc $D$ with the vertices of $N(\{u,v\})$ on the boundary of $D$, and $H_1$ is drawn outside $D$; we may assume that the $K_{3,3}$-subdivision $K$ intersects in two nonadjacent edges $h_1$ and $h_2$ only. Let $d' = d - \kr_{\phi}(d_{\restriction E(H_1)^2})\kr_{\phi}(q_{\restriction E(H_1)^2})q$. Then $\kr_{\phi}(d'_{\restriction E(H_1)^2}) = 0$.

By Lemma~\ref{lem:oncycle}, there exist a $d_1\in B_{u,v}(H_1)$, a $d_2\in B_{v,u}(H_1)$, and integers $\alpha$ and $\beta$ such that 
\begin{align*}
P_{u,v}(d'-d_1-d_2) &= \alpha P_{u,v}(d_K),\\ 
P_{v,u}(d'-d_1-d_2) &= \beta P_{v,u}(d_K). 
\end{align*}
Let $z = d'-d_1-d_2$.

As $\kr_{\phi}(d'_{\restriction E(H_1)^2})=0$, $\kr_{\phi}(z_{\restriction E(H_1)^2}) = 0$. 
Hence $\kr_{\phi}(\alpha P_{u,v}(d_K) + \beta P_{v,u}(d_K)) = 0$, that is, $\alpha \kr_{\phi}(P_{u,v}(d_K))+ \beta  \kr_{\phi}(P_{v,u}(d_K)) = 0$. Since $\kr_{\phi}(P_{u,v}(d_K)) = \kr_{\phi}(h_1,h_2)$, $\kr_{\phi}(P_{v,u}(d_K)) = \kr_{\phi}(h_2,h_1)$, and $\kr_{\phi}(h_1,h_2) = -\kr_{\phi}(h_2,h_1)$, we obtain that $\alpha=\beta$. Hence $z' = z - \alpha d_K$ satisfies $P_{u,v}(z') = 0$ and $P_{v,u}(z') = 0$.
\end{proof}

\begin{lemma}\label{lem:3hastripod}
Let $G=(V,E)$ be a $3$-connected graph and let $e=uv$ be an edge of $G$ such that $G/e$ is $3$-connected. Suppose $N(\{u,v\})$ has exactly three vertices. If there is a tripod in $H$ with feet in $N(\{u,v\})$, then, for any $2$-cycle $d$ on $G$, there exist a quad $2$-cycle $q$, an integer $\alpha$, a Kuratowski $2$-cycle $d_K$ on a $K_{3,3}$-subdivision $K$ in $G$, and an integer $\beta$, such that $z = d - \alpha q - \beta d_K$ satisfies $P_{u,v}(z) = 0$ and $P_{v,u}(z) = 0$.
\end{lemma}
\begin{proof}
Let $\{u_1,u_2,u_3\} = N(\{u,v\})$. Let $TP_1$ be a tripod of $H$ with feet $u_1,u_2,u_3$. Let $H_1 = G[\delta(u)\cup \delta(v)]$ and let $TP_2 = H_1\setminus \{e\}$. Then $TP_1\cup TP_2$ is a quad $Q$. Let $q$ be a quad $2$-cycle on $Q$ with left side $\{u,a\}$, where $a$ is a midpoint of $TP_1$. Let $\phi$ be a drawing of $G$ in the plane that is in generic position such that $H$ is drawn in a disc $D$ with $u_1,u_2,u_3$ on the boundary of $D$, and $H_1$ is drawn outside $D$; we may assume that $H_1$ intersects in two nonadjacent edges $h_1$ and $h_2$ only. Replacing $d$ by $d - \kr_{\phi}(d_{\restriction E(H_1)^2})\kr_{\phi}(q_{\restriction E(H_1)^2})q$ gives a $2$-cycle with $\kr_{\phi}(d_{\restriction E(H_1)^2}) = 0$.

Let $P_1,P_2,P_3$ three internally disjoint paths in $H$ from $u_1,u_2,u_3$ to $a$, respectively, and let $K = G[E(TP_2)\cup E(P_1)\cup E(P_2)\cup E(P_3)]$. Then $K$ is a $K_{3,3}$-subdivision in $G$. Let $d_K$ be a Kuratowski $2$-cycle on $K$ with $d_K(h_1,h_2)=1$. Then 
\begin{align*}
P_{u,v}(d) &= \alpha P_{u,v}(d_K) \text{ for some integer $\alpha$,}\\
P_{v,u}(d) &= \beta P_{v,u}(d_K) \text{ for some integer $\beta$.}
\end{align*}
Since $\kr_{\phi}(d_{\restriction E(H_1)^2}) = \alpha \kr_{\phi}(h_1,h_2) + \beta \kr_{\phi}(h_2,h_1)$ and $\kr_{\phi}(d_{\restriction E(H_1)^2}) = 0$, $\alpha=\beta$.  So $z = d-\alpha d_K$ satisfies $P_{u,v}(z) = 0$ and $P_{v,u}(z) = 0$.
\end{proof}

\section{Decomposing $2$-cycles on a graph}

In the proof of the main theorem, we will use the following lemma.
\begin{lemma} \label{lem:3connected}
Let $G=(V,E)$ be a $3$-connected graph with $|V| > 4$. Then $G$ has an
edge $e$ such that $G/e$ is $3$-connected.
\end{lemma}
A proof of this lemma can be found in \cite{Diestel}.

\begin{theorem}\label{thm:main}
Any $2$-cycle on a graph $G=(V,E)$ is a sum of cycle-pair $2$-cycles, quad $2$-cycles, and Kuratowski $2$-cycles.
\end{theorem}
\begin{proof}
We show this by induction on the number of vertices of $G$. By
Lemma~\ref{lem:1sep} and Lemma~\ref{lem:2sep}, we may assume that $G$
is $3$-connected. The case where $|V|=4$ is easy.

Let $d$ be a $2$-cycle on $G$. We show
that there exist cycle-pair $2$-cycles $d_{C_i,D_i}$, $i=1,\ldots,k$, quad $2$-cycles $q_i$, $i=1,\ldots,\ell$, and Kuratowski
$2$-cycles $d_{H_i}$, $i=1,\ldots,m$ such that
\begin{equation} \label{eq:subtract}
d' = d - \sum_{i=1}^k d_{C_i,D_i} - \sum_{i=1}^\ell q_i - \sum_{i=1}^m d_{H_i}
\end{equation}
satisfies $P_{u,v}(d') = 0$ and $P_{v,u}(d') = 0$.
Once we have shown this, $d'/e$ is a
$2$-cycle on $G/e$. By induction on the number of vertices, $d'/e$ has the required form.
Lemma~\ref{lem:uncontract} then shows that $d'$ belongs to the
group spanned by all $d_{C,D}$, with $C$ and $D$ disjoint oriented cycles of
$G$, all quad $2$-cycles, and all Kuratowski $2$-cycles. Hence, with
\eqref{eq:subtract}, we have shown the theorem.

By Lemma~\ref{lem:3connected}, there exists
an edge $e=uv$ of $G$ such that $G/e$ is $3$-connected. 

Suppose first $N(\{u,v\})$ has at least four vertices. If there exist a $d_1\in B_{u,v}(G)$,  a $d_2\in B_{v,u}(G)$, and quad $2$-cycles $q_i$, $i=1,\ldots,k$, such that $z = d-d_1-d_2-\sum_{i=1}^k q_i$ satisfies $P_{u,v}(z) = 0$ and $P_{v,u}(z) = 0$, then we are done.  If not, then, by Lemma~\ref{lem:cyclenodisjoint}, there exists an oriented cycle $C$ in $H:=G-\{u,v\}$ containing all vertices of $N(\{u,v\})$, and  $N_H(u)$ and $N_H(v)$ cross in $C$, and there are no disjoint paths in $H$ with ends $r_1,r_2\in N(\{u,v\})$ and $s_1,s_2\in N(\{u,v\})$, respectively, such that $r_1,s_1,r_2,s_2$ occur in $C$ in order.
If $H$ has no tripod $TP_1$ with feet in $N(\{u,v\})$, then, by Lemma~\ref{lem:4notripod}, there exist a $d_1\in B_{u,v}(G)$, a $d_2\in B_{v,u}(G)$, a Kuratowski $2$-cycle $d_K$ on a $K_{3,3}$-subdivision $K$ in $G$, and an integer $\alpha$ such that $z = d-d_1-d_2- \alpha d_K$ satisfies $P_{u,v}(z) = 0$ and $P_{v,u}(z) = 0$.
If $H$ has a tripod with feet in $N(\{u,v\})$, then, by Lemma~\ref{lem:4hastripod}, there exist a $d_1\in B_{u,v}(G)$, a $d_2\in B_{v,u}(G)$, a quad $2$-cycle $q$, an integer $\alpha$, a Kuratowski $2$-cycle $d_K$ on a $K_{3,3}$-subdivision $K$ in $G$, and an integer $\beta$, such that $z = d - d_1-d_2 - \alpha q - \beta d_K$ satisfies $P_{u,v}(z) = 0$ and $P_{v,u}(z) = 0$.

Suppose next that $N(\{u,v\})$ has exactly three vertices. If there is no tripod in $H$ with feet in $N(\{u,v\})$, then, by Lemma~\ref{lem:3notripod}, there exist a Kuratowski $2$-cycle $d_K$ on a Kuratowski subgraph $K$ in $G$, and an integer $\alpha$ such that $z = d- \alpha d_K$ satisfies $P_{u,v}(z) = 0$ and $P_{v,u}(z) = 0$. If there is a tripod in $H$ with feet in $N(\{u,v\})$, then, by Lemma~\ref{lem:3hastripod}, there exist a quad $2$-cycle $q$, an integer $\alpha$, a Kuratowski $2$-cycle $d_K$ on a $K_{3,3}$-subdivision $K$ in $G$, and an integer $\beta$, such that $z = d - \alpha q - \beta d_K$ satisfies $P_{u,v}(z) = 0$ and $P_{v,u}(z) = 0$.
\end{proof}

\section{Decomposing skew-symmetric $2$-cycles on a graph}

The proofs of the following two lemmas are similar to the proofs of their corresponding lemmas in the previous section.

\begin{lemma}\label{lem:tocycleultskew}
Let $e = uv$ be an edge of a $3$-connected graph $G$ such that $G/e$ is $3$-connected. Suppose $N(\{u,v\})$ has at least four vertices. Then one of the following holds:
\begin{enumerate}
\item for any skew-symmetric $2$-cycle $d$ on $G$, there exist a $d_1\in B^{\text{skew}}_{u,v}(G)$ and skew-symmetric quad $2$-cycles $q_i-T(q_i)$ on $G$, $i=1,\ldots,k$, such that $P_{u,v}(d-d_1-\sum_{i=1}^k (q_i-T(q_i))) = 0$; or
\item there exists an oriented cycle $C$ in $H:=G-\{u,v\}$ containing all vertices of $N(\{u,v\})$, and  $N_H(u)$ and $N_H(v)$ cross in $C$, and there are no two disjoint $C$-crossing paths in $H$ with ends in $N(\{u,v\})$.
\end{enumerate}
\end{lemma}


\begin{lemma}\label{lem:skewoncycle}
Let $G$ be a graph and let $e=uv$ be an edge of $G$ such that $G-\{u,v\}$ is a cycle $C$.
Suppose $h_1, h_2\in \delta(u)\setminus \{e\}$ and $k_1,k_2\in \delta(v)\setminus \{e\}$ are distinct edges such that the ends $w_1,z_1,w_2,z_2$ of $h_1,k_1,h_2,k_2$ on $C$, respectively, occur in the given order.
If $d$ is a skew-symmetric $2$-cycle, then there exist a $d'\in B^{\text{skew}}_{u,v}(G)$, and an integer $\alpha\in \mathbb{Z}$ such that $P_{u,v}(d+d')=\alpha \mat{h_1}{h_2}{k_1}{k_2}$ and   $P_{v,u}(d+d')=-\alpha \mat{h_1}{h_2}{k_1}{k_2}$.
\end{lemma}

\begin{lemma}\label{lem:skew4notripod}
Let $e=uv$ be an edge of a $3$-connected graph $G=(V,E)$ such that $G/e$ is $3$-connected. Suppose $H := G-\{u,v\}$ has a cycle $C$ containing all vertices of $N(\{u,v\})$, and $N_H(u)$ and $N_H(v)$ cross in $C$. If there are no two disjoint $C$-crossing paths in $H$ with ends in $N(\{u,v\})$, and there is no tripod in $H$ with feet in $N(\{u,v\})$, then, for any skew-symmetric $2$-cycle $d$ on $G$, there exists a $d_1\in B^{\text{skew}}_{u,v}(G)$ such that $P_{u,v}(d-d_1) = 0$.
\end{lemma}
\begin{proof}
Let $d$ be a skew-symmetric $2$-cycle on $G$.
Since $N_H(u)$ and $N_H(v)$ cross in $C$, there exist distinct edges $h_1, h_2\in \delta(u)\setminus \{e\}$ and distinct edges $k_1,k_2\in \delta(v)\setminus \{e\}$ such that the ends $w_1,z_1,w_2,z_2$ of $h_1,k_1,h_2,k_2$ on $C$, respectively, occur in the given order.

As $G/e$ is $3$-connected, there is no $(\leq 2)$-separation $(H_1, H_2)$ of $H$ with $N(\{u,v\})\subseteq V(H_1)$ and $V(H_1)\not=V(H)$. As there are no disjoint paths in $H$ with ends $p_1,p_2\in N(\{u,v\})$ and $q_1,q_2\in N(\{u,v\})$, respectively, such that $p_1,q_1,p_2,q_2$ occur in $C$ in order, and $H$ has no tripod with feet in $N(\{u,v\})$, $H$ can be embedded in the unit disc $D$ with the vertices of $N(\{u,v\})$ on the boundary in the order induced by $C$, by Theorem~\ref{thm:disjointpaths}. Extend the embedding of $H$ in $D$ to a drawing $\phi$ of $G$ in the plane by mapping the vertices $u$ and $v$, and the interiors of the edges incident with $u$ and $v$ in generic position in the complement of $D$. This can be done such that the edges $h_1$ and $k_1$ cross, while $h_i$ and $k_j$ with $i\not=1$ or $j\not=1$ do not cross.

By the previous lemma, there exist a $d'\in B^{\text{skew}}_{u,v}(G)$, and an integer $\alpha\in \mathbb{Z}$ such that $z = d+d'$ satisfies $P_{u,v}(z)=\alpha \mat{h_1}{h_2}{k_1}{k_2}$ and   $P_{v,u}(z)=-\alpha \mat{k_1}{k_2}{h_1}{h_2}$.
  
As the restriction of the drawing $\phi$ to $H$ is an embedding of $H$ in $D$, $\kr_{\phi}(z_{\restriction E(H)^2})=0$.  Let $H' := G[\delta(u)\cup\delta(v)]$. Since $\kr_{\phi}(d) = 0$ for any $2$-cycle $d$ on $G$, $\kr_{\phi}(z_{\restriction E(H')^2}) = 0$. 
Hence $\kr_{\phi}(P_{u,v}(z) +  P_{v,u}(z)) = 0$; that is, $\alpha \kr_{\phi}(h_1,k_1) - \alpha \kr_{\phi}(k_1,h_1) = 0$. Since  $\kr_{\phi}(h_1,k_1) = -\kr_{\phi}(k_1,h_1)$, we obtain that $\alpha=0$. Hence $P_{u,v}(z) = 0$.
\end{proof}

\begin{lemma}\label{lem:skew3notripod}
Let $e=uv$ be an edge of a $3$-connected graph $G=(V,E)$ such that $G/e$ is $3$-connected. Suppose $N(\{u,v\})$ has exactly three vertices. If there is no tripod in $H := G-\{u,v\}$ with feet in $N(\{u,v\})$, then, for any skew-symmetric $2$-cycle $d$ on $G$, $P_{u,v}(d) = 0$.
\end{lemma}
\begin{proof}
As $N(\{u,v\})$ has exactly three vertices, $P_{u,v}(d) = \alpha[f_1,f_2,f_3; g_1, g_2, g_3]$
and $P_{v,u}(d) = -\alpha[g_1,g_2,g_3; f_1, f_2, f_3]$ for some integer $\alpha$, where $f_1,f_2,f_3\in \delta(u)\setminus\{e\}$ and $g_1,g_2,g_3\in \delta(v)\setminus\{e\}$.  Let $u_1,u_2,u_3$ be the three vertices in $N(\{u,v\})$. 

As $G/e$ is $3$-connected, there is no $(\leq 2)$-separation $(H_1, H_2)$ of $H$ with $N(\{u,v\})\subseteq V(H_1)$ and $V(H_1)\not=V(H)$. By Theorem~\ref{thm:disjointpaths}, $H$ can be embedded in a disc $D$ with the vertices of $N(\{u,v\})$ on the boundary. Then $\kr_{\phi}(d_{\restriction E(H)^2}) = 0$. Let $H' := G[\delta(u)\cup\delta(v)]$. Extend the embedding of $H$ in $D$ to a drawing $\phi$ of $G$ in the plane by mapping the vertices $u$ and $v$, and the interiors of the edges incident with $u$ and $v$ in generic position in the complement of $D$. This can be done such that $H'$ has only two nonadjacent edges $h_1,h_2$ that intersect; we assume that $h_1\in \delta(u)$ and $h_2\in \delta(v)$.
Then $\kr_{\phi}(d_{\restriction E(H')^2}) = 0$, by Lemma~\ref{lem:crzero}. That is, 
$\kr_{\phi}(P_{u,v}(d) + P_{v,u}(d)) = \alpha \kr_{\phi}(h_1,h_2) - \alpha \kr_{\phi}(h_1,h_2) = 0$. Since  $\kr_{\phi}(h_1,h_2) = -\kr_{\phi}(h_2,h_1)$, we obtain that $\alpha=0$. Hence $P_{u,v}(d) = 0$.
\end{proof}

Let $\phi$ be a drawing of a graph $G$ in the plane that is in generic position.
Observe that if $d$ is a skew-symmetric $2$-cycle on $G$ and $H$ is a subgraph of $G$, then $\kr_{\phi}(d_{\restriction E(H)^2})$ is even always.

\begin{lemma}\label{lem:skew4hastripod}
Let $e=uv$ be an edge of a $3$-connected graph $G=(V,E)$ such that $G/e$ is $3$-connected. Suppose $C$ is a cycle of $H := G-\{u,v\}$ containing all vertices of $N(\{u,v\})$, and $N_H(u)$ and $N_H(v)$ cross in $C$.  If there are no two disjoint $C$-crossing paths in $H$ with ends in $N(\{u,v\})$, but there is a tripod in $H$ with feet in $N(\{u,v\})$, then, for any skew-symmetric $2$-cycle $d$ on $G$, there exist a $d_1\in B_{u,v}^{\text{skew}}$, a skew-symmetric quad $2$-cycle $q - T(q)$, and an integer $\alpha$ such that $P_{u,v}(d - d_1- \alpha (q - T(q))) = 0$.
\end{lemma}
\begin{proof}
Let $d$ be a skew-symmetric $2$-cycle on $G$. By Lemma~\ref{lem:2tripods},  there exist a tripod $TP_1$ in $H$ with feet, $p_1,p_2,p_3$, in $V(C)$ and a tripod $TP_2$ with feet $p_1,p_2,p_3$ in $G[\delta(u)\cup \delta(v)\cup E(C)]$ such that $V(TP_1\cap TP_2) = \{p_1,p_2,p_3\}$.
Then $TP_1\cup TP_2$ is a quad $Q$ in $G$. Choose midpoints $u_1,u_2$ in $TP_1$ and $TP_2$, respectively. Let $q$ be a quad $2$-cycle on $Q$ with left side $\{u_1,u_2\}$.

Since $N_H(u)$ and $N_H(v)$ cross in $C$, there exist distinct edges $h_1, h_2\in \delta(u)\setminus \{e\}$ and distinct edges $k_1,k_2\in \delta(v)\setminus \{e\}$ such that the ends $w_1,z_1,w_2,z_2$ of $h_1,k_1,h_2,k_2$ on $C$, respectively, occur in the given order.

Let $\phi$ be a drawing of $G$ in the plane that is in generic position such that $H$ is drawn in a disc $D$ with the vertices of $N(\{u,v\})$ on the boundary of $D$, and $H_1$ is drawn outside $D$; we may assume that the $K_{3,3}$-subdivision $K$ intersects in two nonadjacent edges $h_1$ and $h_2$ only. 

Let $$d' = d - \frac{\kr_{\phi}(d_{\restriction E(H_1)^2})}{\kr_{\phi}((q-T(q))_{\restriction E(H_1)^2})}(q-T(q)).$$ Then $\kr_{\phi}(d'_{\restriction E(H_1)^2}) = 0$.

By Lemma~\ref{lem:oncycle}, there exist a $d_1\in B^{\text{skew}}_{u,v}(H_1)$, and an integer $\alpha$ such that 
\begin{align*}
P_{u,v}(d'-d_1) &= \alpha \mat{h_1}{h_2}{k_1}{k_2},\\ 
P_{v,u}(d'-d_1) &= -\alpha \mat{k_1}{k_2}{h_1}{h_2}. 
\end{align*}
Let $z = d'-d_1$.

As $\kr_{\phi}(d'_{\restriction E(H_1)^2})=0$, $\kr_{\phi}(z_{\restriction E(H_1)^2}) = 0$. 
Hence $\kr_{\phi}( \alpha \mat{h_1}{h_2}{k_1}{k_2} -\alpha \mat{k_1}{k_2}{h_1}{h_2}) = 0$. Since $\kr_{\phi}(\mat{h_1}{h_2}{k_1}{k_2}) = \kr_{\phi}(h_1,k_1)$, $\kr_{\phi}(\mat{k_1}{k_2}{h_1}{h_2}) = \kr_{\phi}(k_1,h_1)$, and $\kr_{\phi}(h_1,k_1) = -\kr_{\phi}(k_1,h_1)$, we obtain that $\alpha= 0$. Hence $P_{u,v}(z) = 0$.
\end{proof}

\begin{lemma}\label{lem:skew3hastripod}
Let $G=(V,E)$ be a $3$-connected graph and let $e=uv$ be an edge of $G$ such that $G/e$ is $3$-connected. Suppose $N(\{u,v\})$ has exactly three vertices. If there is a tripod in $H$ with feet in $N(\{u,v\})$, then, for any skew-symmetric $2$-cycle $d$ on $G$, there exist a skew-symmetric quad $2$-cycle $q-T(q)$ and an integer $\alpha$ such that $P_{u,v}(d - \alpha (q-T(q)))$. 
\end{lemma}
\begin{proof}
Let $d$ be a skew-symmetric $2$-cycle on $G$. Let $\{u_1,u_2,u_3\} = N(\{u,v\})$. Let $TP_1$ be a tripod of $H$ with feet $u_1,u_2,u_3$. Let $H' = G[\delta(u)\cup \delta(v)]$ and let $TP_2 = H_1\setminus \{e\}$. Then $TP_1\cup TP_2$ is a quad $Q$. Let $q$ be a quad $2$-cycle on $Q$ with left side $\{u,a\}$, where $a$ is a midpoint of $TP_1$. Let $\phi$ be a drawing of $G$ in the plane that is in generic position such that $H$ is drawn in a disc $D$ with $u_1,u_2,u_3$ on the boundary of $D$, and $H_1$ is drawn outside $D$; we may assume that $H'$ intersects in two nonadjacent edges $h_1$ and $h_2$ only. Let 
$$z := d - \frac{\kr_{\phi}(d_{\restriction E(H_1)^2})}{\kr_{\phi}((q-T(q))_{\restriction E(H_1)^2})}(q-T(q)).$$
Then $\kr_{\phi}(z_{\restriction E(H_1)^2}) = 0$.
Since $P_{u,v}(z) = \alpha[f_1, f_2, f_3; g_1, g_2, g_3]$ and $P_{v,u}(z) = -\alpha[g_1, g_2, g_3; f_1, f_2, f_3]$. Then $\kr_{\phi}(P_{u,v}(z) + P_{v,u}(z)) = \alpha \kr_{\phi}(h_1,h_2) - \alpha \kr_{\phi}(h_2,h_1) = 0$. Since $\kr_{\phi}(h_1,h_2) = -\kr_{\phi}(h_2,h_1)$, we obtain that $\alpha=0$. Hence $P_{u,v}(z)$.
\end{proof}

\begin{theorem}
Any $d\in L^{\text{skew}}(G)$ is a sum of skew-symmetric cycle-pair $2$-cycles and skew-symmetric quad $2$-cycles.
\end{theorem}
\begin{proof}
We show this by induction on the number of vertices of $G$. By
Lemma~\ref{lem:1sep} and Lemma~\ref{lem:2sep}, we may assume that $G$
is $3$-connected. The case where $|V|=4$ is easy.

Let $d$ be a skew-symmetric $2$-cycle on $G$. By Lemma~\ref{lem:3connected}, there exists an edge $e=uv$ of $G$ such that $G/e$ is $3$-connected. 
We show that there exist skew-symmetric cycle-pair $2$-cycles $d_{C_i,D_i} - d_{D_i, C_i}$, $i=1,\ldots,k$ and skew-symmetric quad $2$-cycles $q_i - T(q_i)$, $i=1,\ldots,\ell$ such that
\begin{equation} \label{eq:skewsubtract}
d' = d - \sum_{i=1}^k (d_{C_i,D_i} - d_{D_i, C_i}) - \sum_{i=1}^\ell (q_i - T(q_i)) 
\end{equation}
satisfies $P_{u,v}(d') = 0$.
Then $d'/e$ is a
$2$-cycle on $G/e$. By induction on the number of vertices, $d'/e$ has the required form.
Lemma~\ref{lem:uncontract} then shows that $d'$ belongs to the
group spanned by all skew-symmetric quad $2$-cycles and all skew-symmetric quad $2$-cycles. Hence, with \eqref{eq:skewsubtract}, we have shown the theorem.

Suppose first $N(\{u,v\})$ has at least four vertices. If there are a $d_1\in B^{\text{skew}}_{u,v}(G)$ and skew-symmetric quad $2$-cycles $q_i-T(q_i)$, $i=1,\ldots,k$, such that $P_{u,v}(d-d_1-\sum_{i=1}^k (q_i-T(q_i))) = 0$, then we are done.  If not, then, by Lemma~\ref{lem:cyclenodisjoint}, there exists an oriented cycle $C$ in $H:=G-\{u,v\}$ containing all vertices of $N(\{u,v\})$, and  $N_H(u)$ and $N_H(v)$ cross in $C$, and there are no disjoint paths in $H$ with ends $r_1,r_2\in N(\{u,v\})$ and $s_1,s_2\in N(\{u,v\})$, respectively, such that $r_1,s_1,r_2,s_2$ occur in $C$ in order.
If $H$ has no tripod $TP_1$ with feet in $N(\{u,v\})$, then, by Lemma~\ref{lem:skew4notripod}, there exist a $d_1\in B^{\text{skew}}_{u,v}(G)$ such that $P_{u,v}(d-d_1) = 0$.
If $H$ has a tripod with feet in $N(\{u,v\})$, then, by Lemma~\ref{lem:skew4hastripod}, there exist a $d_1\in B_{u,v}(G)$, a quad $2$-cycle $q$, and an integer $\alpha$ such that $P_{u,v}(d - d_1- \alpha (q-T(q))) = 0$.

Suppose next that $N(\{u,v\})$ has exactly three vertices. If there is no tripod in $H$ with feet in $N(\{u,v\})$, then, by Lemma~\ref{lem:skew3notripod}, $P_{u,v}(d) = 0$. If there is a tripod in $H$ with feet in $N(\{u,v\})$, then, by Lemma~\ref{lem:skew3hastripod}, there exist a skew-symmetric quad $2$-cycle $q-T(q)$, and an integer $\alpha$ such that $P_{u,v}(d - \alpha (q-T(q))) = 0$.
\end{proof}

\section{$2$-cycles on Kuratowski-connected graphs}\label{sec:Kurconn}

Recall that a graph $G$ is Kuratowski connected if no $(\leq 3)$-separation $(G_1,G_2)$ divides Kuratowski subgraphs $H_1$ and $H_2$.


\begin{lemma}\label{lem:quadsim}
Let $K$ be a quad of $G$ with axles $\{a,b\}$ and $\{c,d\}$, and let $q_{a,c}$ be a quad $2$-cycle with left side $\{a,c\}$. If $K$ has width zero, then there exist quad $2$-cycles $q_{a,d}$ and $q_{a,b}$, with left side $\{a,d\}$ and $\{a,b\}$, respectively, such that
$q_{a,b}-q_{a,c}+q_{a,d} = d_H$ for some Kuratowski $2$-cycles $d_H$ on the Kuratowski subgraph $K-\{a\}$. If the width of $K$ is nonzero, then there exists a quad $2$-cycles $q_{a,d}$ with left side $\{a,d\}$, such that $q_{a,c}+q_{a,d} -  d_H\in B(K)$ for some Kuratowski $2$-cycles $d_H$ on the Kuratowski subgraph $H$ obtained from $K$ by deleting all branches of $K$ that contain $a$.
\end{lemma}
\begin{proof}
We show that $q_{a,b}-q_{a,c}+q_{a,d} = d_H$ for some Kuratowski $2$-cycles $d_H$ on the Kuratowski subgraph $K-\{a\}$, if $K$ has width zero, leaving the other statement to the reader.

Denote by $r_1,r_2,r_3$ the vertices of degree four in $K$. For $i=1,2,3$, let $a_i, b_i, c_i, d_i$ be the edges connecting $r_i$ and $a$, connecting $r_i$ and $b$, connecting $r_i$ and $c$, and connecting $r_i$ and $d$, respectively. Let $A=\{a_1,a_2,a_3\}$, let $B=\{b_1,b_2,b_3\}$, let $C=\{c_1,c_2,c_3\}$, and let $D=\{d_1,d_2,d_3\}$. Then we can write $q_{a,b}, q_{a,c}, q_{a,d}$ as
\begin{equation*}
q_{a,b} = \bordermatrix{~ & A & B & C & D \cr
A & 0 & 0 & U & -U\cr
B & 0 & 0 & -U & U\cr
C & 0 & 0 & 0 & 0\cr
D & 0 & 0 & 0 & 0\cr},
q_{a,c} = \bordermatrix{~ & A & B & C & D \cr
A & 0 & U & 0 & -U\cr
B & 0 & 0 & 0 & 0\cr
C & 0 & -U & 0 & U\cr
D & 0 & 0 & 0 & 0\cr},
q_{a,d} = \bordermatrix{~ & A & B & C & D \cr
A & 0 & U & -U & 0\cr
B & 0 & 0 & 0 & 0\cr
C & 0 & 0 & 0 & 0\cr
D & 0 & -U & U & 0\cr}.
\end{equation*}
Then
\begin{equation*}
q_{a,b}-q_{a,c}+q_{a,d} = \bordermatrix{~ & A & B & C & D \cr
A & 0 & 0 & 0 & 0\cr
B & 0 & 0 & -U & U\cr
C & 0 & U & 0 & -U\cr
D & 0 & -U & U & 0\cr}.
\end{equation*}
Since $U^T = -U$, we see that $q_{a,b}-q_{a,c}+q_{a,d}$ is a symmetric $2$-cycle on 
the subgraph $H$ of $K$ obtained by deleting the branches containing vertex $a$. As $H$  is a subdivision of a Kuratowski subgraph, $q_{a,b}-q_{a,c}+q_{a,d}=d_H$ for a Kuratowski $2$-cycle $d_H$ on $H$.
\end{proof}

\begin{lemma}\label{lem:quadkuratowskiconnected}
Let $q$ be a quad $2$-cycle on a graph $G$. If $G$ is Kuratowski connected, then either $q\in B(G)$ or $q - d_H\in B(G)$ for some Kuratowski $2$-cycle $d_H$. 
\end{lemma}
\begin{proof}
Suppose to the contrary that there is a quad $2$-cycle $q$ on a quad $K_q$ of $G$ such that neither $q\in B(G)$ nor $q - d_H\in B(G)$ for some Kuratowski $2$-cycle $d_H$. Choose such a quad $2$-cycle $q$ for which the width of $K_q$ is minimum and let 
\begin{equation*}
K_q = P_1\cup P_2\cup P_3\cup Q_1\cup Q_2\cup Q_3\cup R_1\cup R_2\cup R_3,
\end{equation*} 
where $P_1, P_2, P_3, Q_1, Q_2, Q_3, R_1, R_2, R_3$ are as in the definition of quad. Let $\{a,b\}$ and $\{c,d\}$ be the axles and $\{a,c\}$ be the left side of the quad. For $i=1,2,3$, let $p_i$ be the vertex in $V(P_i\cap Q_i)$ and let $r_i$ be the vertex in $V(R_i\cap Q_i)$. 
For $i=2,3$, let $C_i = P_{L,1}\cup Q_1\cup R_{L,1}\cup R_{L,i}\cup Q_i\cup P_{L,i}$ and let $D_i = P_{R,1}\cup Q_1\cup R_{R,1}\cup R_{R,i}\cup Q_i\cup P_{R,i}$. Orient the edges in the paths $P_{L,1}$ and $P_{R,1}$ towards $p_1$.
Orient $C_2, C_3$ such that $P_{L,1}$ is traversed in forward direction, and orient $D_2, D_3$ such that $P_{R,1}$ is traversed in forward direction.
Then $$q = d_{C_2, D_3} - d_{C_3, D_2}.$$

We first show that  there is no path in $G$ with an end in $V(P_1\cup P_2\cup P_3)\setminus V(Q_1\cup Q_2\cup Q_3)$ and an end in $V(Q_1\cup Q_2\cup Q_3)\setminus V(P_1\cup P_2\cup P_3)$ that is internally disjoint from $K_q$.

Suppose first that there is a path $P$ in $G$ with an end in $V(P_{L,i}) \setminus V(Q_i)$ and an end in $V(Q_i)\setminus V(P_i)$ for some $i\in \{1,2,3\}$ that is internally disjoint from $K_q$; we may assume that $i=1$. Let $C$ be the cycle in $P_{L,1}\cup Q_1\cup P$ and let $D$ be the cycle in $P_{R,2}\cup P_{R,3}\cup Q_2\cup Q_3\cup R_{R,2}\cup R_{R,3}$. Let $e$ be the edge of $P_{L,1}$ incident with $p_1$, and let $f$ be the edge of $P_{R,2}$ incident with $p_2$. Orient $C$ and $D$ such that $d_{C,D}(e,f) = 1$.
Then $q' = q- q(e,f)d_{C,D}$ is a quad $2$-cycle on a quad whose width is smaller than the width of $K_q$.

Suppose next that there is a path $P$ in $G$ with an end in $V(P_{L,i}) \setminus V(Q_i)$ and an end in $V(Q_j)\setminus V(P_j)$ for some $i,j \in \{1,2,3\}$ with $i\not=j$ that is internally disjoint from $K_q$; we may assume that $i=1$ and $j=2$. Let $C$ be the cycle in $P_{L,1}\cup P_{L,2}\cup Q_2\cup P$ and let $D$ be the cycle in $P_{R,1}\cup P_{R,3}\cup Q_1\cup Q_3\cup R_{R,1}\cup R_{R,3}$. Let $e$ be the edge of $Q_2$ incident with $p_2$, and let $f$ be the edge of $P_{R,1}$ incident with $p_1$. Orient $C$ and $D$ such that $d_{C,D}(e,f) = 1$. Then $q'=q-q(e,f)d_{C,D}$ is a quad $2$-cycle on a quad whose width is smaller than the width of $K_q$. 

Therefore there is no path with an end in $V(P_1\cup P_2\cup P_3)\setminus V(Q_1\cup Q_2\cup Q_3)$ and an end in $V(Q_1\cup Q_2\cup Q_3)\setminus V(P_1\cup P_2\cup P_3)$ that is internally disjoint from $K_q$.

Suppose next that there is a path $P$ with an end in $V(P_{L,i}) \setminus V(Q_i)$ and an end in $V(R_{L,j}) \setminus V(Q_j)$ for some $i,j\in \{1,2,3\}$ that is internally disjoint from $K_q$. We assume that $i=1$ and $j=2$; the other cases are similar. Let $C'_2$ be the cycle in $P_{L,1}\cup Q_1\cup R_{L,1}\cup R_{L,2}\cup P$ and let $C''_2$ be the cycle in $P_{L,1}\cup P \cup R_{L, 2}\cup Q_2\cup P_{L, 2}$. Orient $C'_2, C''_2$ such that the edges of $P_{L,1}$ are traversed in forward direction. Then $\chi_{C'_2} + \chi_{C''_2} = \chi_{C_2}$. In the same way, we define $C'_3$ and $C''_3$. Then $\chi_{C'_3} + \chi_{C''_3} = \chi_{C_3}.$ Then
$$q = d_{C_2, D_3} - d_{C_3, D_2} = d_{C'_2, D_3} + d_{C''_2, D_3} - d_{C'_3, D_2} -d_{C''_3, D_2} \in B(G).$$ Therefore, there is no path with an end in $V(P_{L,i}) \setminus V(Q_i)$ and an end in $V(R_{L,j}) \setminus V(Q_j)$ for some $i,j\in \{1,2,3\}$ that is internally disjoint from $K_q$. Similarly, there is no path with an end in $V(P_{R,i}) \setminus V(Q_i)$ and an end in $V(R_{R,j}) \setminus V(Q_j)$ for some $i,j\in \{1,2,3\}$ that is internally disjoint from $K_q$.  

Since $G$ is Kuratowski connected, a path $P$ must exist with an end in $V(P_{L,i}) \setminus V(Q_i)$ and an end in $V(R_{R,j}) \setminus V(Q_j)$ or with an end in $V(P_{R,i}) \setminus V(Q_i)$ and an end in $V(R_{L,j}) \setminus V(Q_j)$.
If $P$ has an end in $V(P_{L,i}) \setminus V(Q_i)$ and an end in $V(R_{R,j}) \setminus V(Q_j)$, and the width of $K_q$ is nonzero, then Lemma~\ref{lem:quadsim} 
brings us back to the case where there is a path with an end $u$ in $V(P_{L,i}) \setminus V(Q_i)$ and an end $v$ in $V(R_{L,j}) \setminus V(Q_j)$ for some $i,j\in \{1,2,3\}.$ The case where $P$ has an end in $V(P_{R,i}) \setminus V(Q_i)$ and an end in $V(R_{L,j}) \setminus V(Q_j)$ is similar. We may therefore assume that the width of $K_q$ is zero. 
Since $G$ is Kuratowski connected a path $Q$ must exist with an end in $V(P_{L,i}) \setminus V(Q_i)$ and an end in $V(P_{R,i}) \setminus V(Q_i)$, or an end in $V(R_{L,i}) \setminus V(Q_i)$ and an end in $V(R_{R,i}) \setminus V(Q_i)$. Also in this case, Lemma~\ref{lem:quadsim} brings us back to the case where there is a path with an end in $V(P_{L,i}) \setminus V(Q_i)$ and an end in $V(R_{L,j}) \setminus V(Q_j)$ for some $i,j\in \{1,2,3\}.$
\end{proof}

\begin{theorem}\label{thm:mainKur}
Let $G$ be a Kuratowski-connected graph. Then $L(G)$ is generated by cycle-pair $2$-cycles and Kuratowski $2$-cycles, if any.
\end{theorem}

For skew-symmetric $2$-cycles we have the following.
\begin{lemma}\label{lem:quadsimskew}
Let $K$ be a quad of $G$ with axles $\{a,b\}$ and $\{c,d\}$, and let $q_{a,c}$ be a quad $2$-cycle with left side $\{a,c\}$. If $K$ has width zero, then there exist skew-symmetrc quad $2$-cycles $q_{a,d}$ and $q_{a,b}$, with left side $\{a,d\}$ and $\{a,b\}$, respectively, such that $(q_{a,c}-T(q_{a,c}))+(q_{a,d}-T(q_{a,d}))+(q_{a,b}-T(q_{a,b})) = 0$. If the width of $K$ is nonzero, then there exists a quad $2$-cycles $q_{a,d}$ such that $(q_{a,c}-T(q_{a,c}))+(q_{a,d}-T(q_{a,d})) \in B^{\text{skew}}(K)$.
\end{lemma}

A proof similar to the proof of Lemma~\ref{lem:quadkuratowskiconnected} shows the following.
\begin{lemma}\label{lem:skewquadkuratowskiconnected}
Let $q$ be a skew-symmetric quad $2$-cycle on a graph $G$. If $G$ is Kuratowski connected, then $q\in B^{\text{skew}}(G)$. 
\end{lemma}

The following lemma can be shown by inspection.
\begin{lemma}\label{lem:Petersen}
For each graph $G$ in the Petersen family, $d_H\in B(G)$ for some Kuratowski $2$-cycle $d_H$ on $G$.
\end{lemma}

\begin{lemma}\label{lem:decontract2cycle}
Let $G'$ be a minor of $G$. If there is a Kuratowski $2$-cycle $d_{H'}$ such that
$d_{H'} \in B_{G'}$, then there is a Kuratowski $2$-cycle $d_H$ such that
$d_H \in B_G$.
\end{lemma}
\begin{proof}
We may assume that $G'$ arises from $G$ by contracting one edge $e$.
We can write $d_{H'} = \sum_{i} \alpha_i d_{C'_i,D'_i}$. By
Lemma~\ref{lem:uncontract}, there is a Kuratowski subgraph $H$ of
$G$, a Kuratowski $2$-cycle $d_{H}$ on $H$, a pair of disjoint
oriented cycles $C,D$ of $G$, and an $\alpha\in \{0,1\}$, such
that $d' = d_{H} + \alpha d_{C,D}$ satisfies $d'/e = d_{H'}$. By the
same lemma there are pairs of disjoint oriented cycles $C_i,D_i$
such that $d_{C_i,D_i}/e = d_{C'_i,D'_i}$. Let $d = d_{H} +
\alpha d_{C,D} - \sum_{i} \alpha_i d_{C_i,D_i}$. Then $d(f_1,f_2) =
0$ for each pair of edges $f_1,f_2$ in $G\setminus e$. Hence, $d(f_1,f_2) = 0$ for each pair of edges $f_1,f_2$ in $G$.
\end{proof}

Let $\Gamma$ be a frame in $\oR^3$ isomorphic to $G$ and consider a diagram of $\Gamma$ into some plane; see \cite{Robertson2} for the definition of a frame in $\oR^3$. The sign of a crossing is defined as in
Figure~\ref{fig:signcrossing}.
\begin{figure}
\begin{center}
\includegraphics[width=0.3\textwidth]{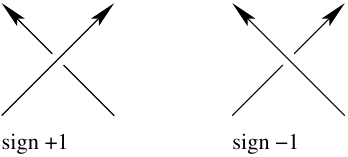}
\caption{The sign of a crossing}\label{fig:signcrossing}
\end{center}
\end{figure}
For $e,f \in E$, define $\sign_{\Gamma}(e,f)$ as the sum of
the signs of all crossings of $e$ with $f$.
For a $2$-cycle $d$, define the \emph{linking number} of
$d$ by
\begin{equation*}
\link_{\Gamma} (d) = \frac{1}{2}\sum_{e, f \in E} \sign_{\Gamma}(e, f) d(e, f).
\end{equation*}
The linking number is independent of the chosen diagram of the
embedding. If $C$ and $D$ are disjoint oriented circuits of $G$,
then $\link_{\Gamma}(d_{C,D})$ equals the linking number of $C$ and $D$ when viewed as cycles in $\Gamma$. If $d_1$ and $d_2$ are $2$-cycles and $d = d_1+d_2$, then $\link_{\Gamma}(d) = \link_{\Gamma}(d_1) + \link_{\Gamma}(d_2)$. Observe that, for any Kuratowski $2$-cycle $d_H$, $\link_{\Gamma}(d_H)$ is odd; see \cite{Holst2003b}.

A \emph{linkless embedding} of a graph $G$ is an embedding of $G$ such that each pair of disjoint cycles $C, D$ has linking number $\link(C, D) = 0$. 
The Petersen family is the family of all graphs that can be obtained from $K_6$, the complete graph on six vertices, by applying $\Delta Y$-transformations and $Y\Delta$-transformations. One graph in this family of graphs is the Petersen graph. Robertson, Seymour, and Thomas \cite{MR1339849} proved that a graph has a linkless embedding in $3$-space if and only if $G$ has no minor isomorphic to a graph in the Petersen family.

\begin{lemma}
Let $G$ be a Kuratowski-connected graph containing a Kuratowski subgraph. Then the following are equivalent:
\begin{enumerate}
\item\label{item:1} $G$ has a minor isomorphic to a graph in the Petersen family;
\item\label{item:3} $d_H\in B(G)$ for any Kuratowski $2$-cycles $d_H$ of $G$.
\item\label{item:2} $d_H\in B(G)$ for some Kuratowski $2$-cycles $d_H$ of $G$;

\end{enumerate}
\end{lemma}
\begin{proof}
(\ref{item:1})$\implies$ (\ref{item:3}) Suppose $G$ has a minor $P$ isomorphic to a graph in the Petersen family. Let $H$ be a Kuratowski subgraph in $P$. Then $d_H\in B(P)$. Since $G$ is Kuratowski connected, $d_{H'}\in B(G)$ for any Kuratowski $2$-cycle $d_{H'}$ of $G$. 

(\ref{item:3})$\implies$ (\ref{item:2}) Clear.

(\ref{item:2})$\implies$ (\ref{item:1}) Suppose $d_H\in B(G)$ for some Kuratowski $2$-cycles $d_H$ of $G$. Suppose for a contradiction that $G$ has a linkless embedding $\Gamma$. Then, as $\link_{\Gamma}(d_H)$ is odd, there exists a cycle-pair $2$-cycle $d_{C,D}$ with $\link_{\Gamma}(d_{C,D})$ odd. This contradicts that $\Gamma$ is a linkless embedding. Hence $G$ has a minor isomorphic to a graph in the Petersen family.
\end{proof}

\begin{theorem}
Let $G$ be a Kuratowski-connected graph. Then $L(G)=B(G)$ if and only if $G$ is planar or $G$ does not admit a linkless embedding.
\end{theorem}
\begin{proof}
Suppose $L(G)=B(G)$ and $G$ is not planar. Then $G$ has a Kuratowski subgraph $H$ of $G$. As $L(G)=B(G)$, $d_H \in B(G)$.  If $G$ admits a linkless embedding $\Gamma$, then, as $\link_{\Gamma}(d_{C,D}) = 0$ for every pair of disjoint cycles $C, D$, $\link_{\Gamma}(d_H) = 0$. However, this contradicts that $\link_{\Gamma}(d_H)$ is odd. 

If $G$ is planar, then, by the previous theorem, $L(G) = B(G)$. If $G$ does not admit a linkless embedding, then $G$ has a minor $K$ isomorphic to a graph in the Petersen family. By Lemma~\ref{lem:Petersen}, $d_{H}\in B(K)$ for some Kuratowski $2$-cycle $d_H$ on $K$. By Lemma~\ref{lem:decontract2cycle}, $d_{H'}\in B(G)$ for some Kuratowski $2$-cycle $d_{H'}$ on $G$. As $G$ is Kuratowski connected, $d_{H_1}\in B(G)$ for any Kuratowski $2$-cycle $d_{H_1}$ on $G$.
\end{proof}

If $G$ is  a Kuratowki-connected graph that admits a linkless embedding $\Gamma$, then 
\begin{equation}\label{eq:linkless}
B(G) = \{d\in L(G)~:~\link_{\Gamma}(d) = 0\}.
\end{equation}

\begin{theorem}[\cite{Holst2003b}]\label{thm:linkless}
Let $G$ be a graph admitting a linkless embedding. Then there exists a polynomial-time algorithm to find a linkless embedding $\Gamma$ of the graph $G$.
\end{theorem}

\begin{theorem}\label{thm:genset}
Let $G$ be a Kuratowski-connected graph. Then there exists a polynomial-time algorithm to find a set of generators for the space generated by all cycle-pair $2$-cycles.
\end{theorem}
\begin{proof}
We first check whether $G$ admit a linkless embedding or not. If not, then $B(G) = L(G)$, and we can take a generating set for $L(G)$. If $G$ admits a linkless embedding, then we embed $G$ linklessly. Using $(\ref{eq:linkless})$, we can find a generating set for $B(G)$.
\end{proof}

Let $A$ be an abelian group and let $a\in A$. If $n$ is a positive integer, we define $na$ to be the result of adding $a$ $n$ times. Using this, we define $na = (-n)(-a)$ for negative integers $n$. For $n=0$, $0 a = 0$.

For a graph $G=(V, E)$, let $F : E^2\to A$. If $x, y\in \mathbb{Z}^E$, we define $F(x, y) = \sum_{e,f\in E} x(e)y(f)F(e,f)$.

\begin{theorem}
Let $G=(V,E)$ be a Kuratowski-connected graph and let $A$ be an abelian group.
Let $F : E^2\to A$. Then there exists a polynomial-time algorithm for testing whether there exist two disjoint oriented cycles $C,D$ in $G$ such that $F(C,D) \not= 0$. 
\end{theorem}
\begin{proof}
By Theorem~\ref{thm:genset}, we can find a generating set $\mathcal{G}$ of $B(G)$. For all pairs $x,y\in \mathcal{G}$, we check whether $F(x,y) \not=0$. If there exist generators $x,y$ such that $F(x,y) \not=0$, then there exists two disjoint oriented cycles $C, D$ in $G$ such that $F(C, D)\not=0$. If not, there are no two disjoint oriented cycles $C, D$ in $G$ such that $F(C, D)\not=0$.
\end{proof}

In case $A=\mathbb{Z}$ and $F \equiv 1$, the algorithm finds two disjoint cycles $C, D$. In case $A=\mathbb{Z}_m$ $(m>1)$ and $F\equiv 1$, the algorithm find two disjoint cycles $C, D$ with sizes not divisible by $m$. In case $A=\mathbb{Z}$, $F(s_1t_1, s_2t_2)=1$ and $F\equiv 0$ otherwise, the algorithm find two disjoint cycles $C, D$, with $C$ using the edges $s_1t_1$ and $D$ using the edge $s_2t_2$. This corresponds to two disjoint paths in $G-\{s_1t_1,s_2t_2\}$, one joining $s_1$ to $t_1$ and another joining $s_2$ to $t_2$. In case $A=\mathbb{Z}$, $G$ is a spatially embedded graph, $F(e,f) = 1$ if $e$ goes over $f$ with $f$ going from left to right, and $F(e,f)=-1$ if $e$ goes over $f$ with $f$ going from right to left, then the algorithm find two disjoint cycles $C,D$ such that the linking number of $C, D$ is nonzero.


\bibliographystyle{plain}
\bibliography{./biblio}
\end{document}